\documentclass[12pt]{amsart}
\usepackage{amssymb,amsthm,amsmath,pb-diagram,enumerate}
\usepackage[OT2,T1]{fontenc}
\usepackage{eucal,mathrsfs,esvect}
\usepackage{color}
\usepackage[all]{xy}
\usepackage{fullpage}
\usepackage{url}
\usepackage{mathabx}

\makeatletter
\def\subsection{
	\@startsection
	{subsection}	%name
	{2}		%level
	{\z@}		%indent
	{3ex}		%beforeskip
	{1.5ex}		%afterskip
	{\bf}		%style
}
\makeatother

\makeatletter
\def\subsubsection{
	\@startsection
	{subsubsection}	%name
	{3}		%level
	{\z@}		%indent
	{2.4ex}		%beforeskip
	{1.5ex}		%afterskip
	{\it}		%style
}
\makeatother

\theoremstyle{plain}
\newtheorem{lem}{Lemma}[subsection]
\newtheorem{thm}[lem]{Theorem}
\newtheorem*{thm*}{Theorem}
\newtheorem{cor}[lem]{Corollary}
\newtheorem{prop}[lem]{Proposition}

\theoremstyle{definition}
\newtheorem{defn}[lem]{Definition}
\newtheorem{rem}[lem]{Remark}
\newtheorem{notn}[lem]{Notation}

\newcommand{\mbb}[1]{\mathbb #1}

\newcommand{\mc}[1]{\mathcal #1}

\newcommand{\oper}[1]{\operatorname{#1}}

\newcommand{\Gm}{\mbb G_{\mathrm m}}
\newcommand{\Ga}{\mbb G_{\mathrm a}}
\newcommand{\bmu}{\mu}
\newcommand{\Z} {\mbb Z}

\newcommand{\Br}{\oper{Br}}

\newcommand{\Gal}{\oper {Gal}}
\newcommand{\Hom}{\oper{Hom}}
\newcommand{\Aut}{\oper{Aut}}
\newcommand{\cha}{\oper{char}}
\newcommand{\Spec}{\oper{Spec}}
\newcommand{\s} {^{\oper{sep}}}
\newcommand{\gd} {^{\oper{gd}}} %globally dominated \'etale topology
\newcommand{\GL} {\oper{GL}}
\newcommand{\SL} {\oper{SL}}
\newcommand{\Tors}{\oper{Tors}}
\newcommand{\PP}{{\mc P\mc P}}
\newcommand{\Ind}{\oper{Ind}}
\newcommand{\Vect}{\oper{Vect}}
\newcommand{\res}{\oper{res}}
\newcommand{\cores}{\oper{cor}}

\newcommand{\Nrd}{\oper{Nrd}}
\newcommand{\mX}{{\wh X}}

\newcommand{\cX}{{X}}

\renewcommand{\P}{{\mc{P}}}
\newcommand{\UU}{{\mc{U}}}
\newcommand{\B}{{\mc{B}}}

\newcommand{\Fv} {{\prod_{v \in \mc V} F_v}}
\newcommand{\Fe} {{\prod_{e \in \mc E} F_e}}
\newcommand{\Av} {{\prod_{v \in \mc V} A(F_v \otimes_F F\s)}}
\newcommand{\Ae} {{\prod_{e \in \mc E} A(F_e \otimes_F F\s)}}

\newcommand{\ov}{\overline}
\newcommand{\til}{\widetilde}
\newcommand{\wh}{\widehat}
\newcommand{\zmod}[2]{{{\mathbb Z}/{#1}\mathbb Z ({#2})}}

\usepackage[OT2,T1]{fontenc}
\DeclareSymbolFont{cyrletters}{OT2}{wncyr}{m}{n}
\DeclareMathSymbol{\Sha}{\mathalpha}{cyrletters}{"58}

\title{Local-global principles for Galois cohomology}
\author{David Harbater}
\author{Julia Hartmann}
\author{Daniel Krashen}

\date{March 24, 2013}

\begin{document}

\thanks{The first author was supported in part by NSF grant DMS-0901164. 
The second author was supported by the German Excellence Initiative via RWTH Aachen University and by the German National Science Foundation (DFG).
The third author was supported in part by NSF grant DMS-1007462.\\
\textit{Mathematics Subject Classification} (2010): 11E72, 13F25, 14H25 (Primary); 12G05, 20G15 (Secondary).\\  
\textit{Key words and phrases}: local-global principles, Galois cohomology, arithmetic curves, cohomological invariants, patching.}

%-----------------------------------------------------------------------------
\begin{abstract}
This paper proves local-global principles for Galois cohomology groups over function fields $F$ of curves that are defined over a complete discretely valued field.  We show in particular that such principles hold for $H^n(F, \Z/m \Z(n-1))$, for all $n>1$.  This is motivated by work of Kato and others, where such principles were shown in related cases for $n=3$.  Using our results in combination with cohomological invariants, we obtain local-global principles for torsors and related algebraic structures over $F$.  Our arguments rely on ideas from patching as well as the Bloch-Kato conjecture.
\end{abstract}

\maketitle

%%%%%%%%%%%%%%%%%%%
%%%%%%%%%%%%%%%%%%%
\section{Introduction}
%%%%%%%%%%%%%%%%%%%
%%%%%%%%%%%%%%%%%%%

In this paper we present local-global principles for Galois cohomology,
which may be viewed as higher-dimensional generalizations of classical
local-global principles for the Brauer group.  
These results then lead to 
local-global principles for other algebraic structures as well, via cohomological invariants.

Recall that if $F$ is a global
field, the theorem of Albert-Brauer-Hasse-Noether  
says a central simple $F$-algebra is isomorphic to a matrix algebra if and only if this is true over each 
 completion~$F_v$ of $F$.  Equivalently, the  
natural group homomorphism
\[\Br(F) \to \prod_{v \in \Omega_F} \Br(F_v)\]
is injective, where $\Omega_F$ is the set of places of $F$.  

Kato suggested a higher dimensional generalization of this in~\cite{Kato},  drawing on the observation that the above result provides a local-global principle for the 
$m$-torsion part of the Brauer group $\Br(F)[m] = H^2(F,\Z/m \Z(1))$.  (Here $\Z/m \Z(n)$ denotes $\bmu_m
^{\otimes n}$, for $m$ not divisible by $\cha(F)$.)  He proposed that the natural domain for higher-dimensional
versions of local-global principles 
should be $H^n(F, \Z/m \Z(n-1))$, for $n>1$. 
Cohomological invariants (such as the Rost invariant) often take values in $H^n(F, \Z/m \Z(n-1))$ for some $n>1$; and thus such local-global principles for cohomology could be used to obtain 
local-global principles for other algebraic objects.

In Theorem~0.8(1) of \cite{Kato}, 
Kato proved such a principle with $n=3$ 
for the function field $F$ of a smooth 
proper surface $X$ over a finite field, both with respect to the discrete valuations on $F$ that arise from codimension one points on $X$, and alternatively with respect to the set of closed points of $X$ (in the latter case using the fraction fields of the complete local rings at the points). 
He also proved a related result \cite[Theorem~0.8(3)]{Kato} for arithmetic surfaces, i.e.\ for curves over rings of integers of number fields.  The corresponding assertions for $n>3$ are vacuous in his situation, for cohomological dimension reasons; and the analogs for $n=2$ do not hold there (e.g.\ if the unramified Brauer group of the surface is non-trivial).

Unlike the classical case of dimension one, in dimension two it is also meaningful to consider local-global principles for fields that are not global, e.g.\ $k((x,y))$ or $k((t))(x)$.  In~\cite[Theorem~3.8]{COP},  
the authors start with an irreducible surface over a finite field of characteristic not dividing $m$; and they take 
the fraction field $F$ of the henselization of the local ring at a closed point.  In that situation, they prove a local-global principle for 
$H^3(F, \Z/m \Z(2))$ with respect to the discrete valuations on $F$.
Also, while not explicitly said in~\cite{Kato}, it is possible to use Theorem~5.2 of that paper
to obtain a local-global principle for function fields $F$ of curves over a non-archimedean local field, with respect to 
$H^3(F,\Z/m\Z(2))$.  This was relied on in \cite[Theorem~5.4]{CPS} and~\cite{Hu} (cf.\ also \cite[pp.~139 and 148]{PS}).

%%%%%%%%%%%%%%%%%%%
\subsection{Results}
%%%%%%%%%%%%%%%%%%%

In this manuscript, we show that when $F$ is the function field of a
curve over an {\em arbitrary} complete discretely valued field $K$,
local-global principles hold for the cohomology groups $H^n(F, \Z/m
\Z(n-1))$ for {\em all} $n>1$. 

In particular we obtain the following local-global principle with
respect to points on the closed fiber $\cX$ of a model $\mX$ of $F$ over the
valuation ring of $K$ (where $k$ is the residue field):

\begin{thm*}[\ref{shapoints}]
Let $n>1$ and let $A$ be one of the following algebraic groups over $F$:
\renewcommand{\theenumi}{\roman{enumi}}
\renewcommand{\labelenumi}{(\roman{enumi})}
\begin{enumerate}
\item 
$\Z/ m \Z (n-1)$, where $m$ is not divisible by the
characteristic of $k$, or
\item 
$\Gm$, if $\cha(k)=0$ and $K$ contains a primitive $m$-th root
of unity for all $m \ge 1$.
\end{enumerate}
Then  the natural map
\[H^n(F, A) \to \prod_{P \in \cX} H^n(F_P, A)\]
is injective, where $P$ ranges through all the points of the closed fiber $\cX$.
\end{thm*}

Here $F_P$ denotes the fraction field of the complete local ring of $\mX$ at $P$.

We also obtain a local-global principle with respect to discrete valuations
if $K$ is equicharacteristic:
\begin{thm*}[\ref{sha_vanish}]
Suppose that $K$ is an equicharacteristic complete discretely valued field of
characteristic not dividing $m$, and that $\mX$ is a regular projective
$T$-curve with function field $F$. 
Let $n>1$.  Then the natural map
\[H^n(F, \zmod m {n-1}) \to \prod_{v \in \Omega_\mX} H^n(F_v, \zmod m
{n-1})\]
is injective.
\end{thm*}

Here $\Omega_\mX$ is the set of discrete valuations on $F$ that arise from codimension one points on $\mX$.  Also, in the above results and henceforth, the cohomology that is used is Galois cohomology, where $H^n(F,A) = H^n(\Gal(F),A(F\s))$ for $A$ a smooth commutative group scheme over $F$ and $n \ge 0$, with $H^0(F,A)=A(F)$.  (For non-commutative group schemes, we similarly have $H^0$ and $H^1$.)

These results also yield new local-global principles for torsors under
linear algebraic groups by the use of cohomological invariants such
as the Rost invariant (\cite[p.~129]{GMS}), following a strategy used in~\cite{CPS} and \cite{Hu}.  We list some of these applications of our
local-global principles in Section~\ref{applications}.
Note that although we also obtained certain local-global
principles for torsors for linear algebraic groups in~\cite{HHK:H1},
the results presented here use a different set of hypotheses on the
group.  In particular, here we do not require that the group $G$ be
rational, unlike in \cite{HHK:H1}.

%%%%%%%%%%%%%%%%%%%
\subsection{Methods and structure of the manuscript}
%%%%%%%%%%%%%%%%%%%

Our approach to obtaining these local-global principles uses the
framework of patching over fields, as in~\cite{HH}, \cite{HHK},
and~\cite{HHK:H1}. The innovation is that these principles derive from
long exact Mayer-Vietoris type sequences with respect to the
``patches'' that arise in this framework.  These sequences are
analogous to those in~\cite{HHK:H1} for linear algebraic groups that
were not necessarily commutative (but where only $H^0$ and $H^1$ were
considered for that reason).  

In Section~\ref{abstract patching}, we derive Mayer-Vietoris
sequences and local-global principles in an abstract context of a
field together with a finite collection of overfields (Section~\ref{abstract MV}). This allows us
to isolate the necessary combinatorial, group-theoretic, and cohomological properties of our
fields and Galois modules. The combinatorial data of
the collection of fields we use is encoded in the notion of a $\Gamma$-field; see
Section~\ref{gamma fields}.  The key group-theoretic property of our
Galois modules is ``separable factorization'', introduced in
Section~\ref{sep fact sect}.  The cohomological properties we
require are formulated in the concept of global domination of Galois
cohomology (Sections~\ref{gl dom sect} and~\ref{gl dom cond sect}).
An essential ingredient in our arguments is the Bloch-Kato conjecture.

In Section~\ref{curve section}, we apply our results to the situation of a
function field over a complete discretely valued field. In
Section~\ref{curve patches} we obtain a local-global principle with
respect to ``patches.'' This is used in Section~\ref{curve points} to
obtain a local-global principle with respect to points on the closed
fiber of a regular model. Finally, in Section~\ref{curve dvr}, we
obtain our local-global principle with respect to discrete valuations
with the help of a result of Panin \cite{Panin:EGC} for local rings
in the context of Bloch-Ogus theory. This step is related to ideas
used in~\cite{Kato}.

In Section~\ref{applications}, we combine our local-global principles
with cohomological invariants taking values in $H^n(F, \Z/m
\Z(n-1))$, to obtain our applications to other algebraic structures.

\medskip

\noindent{\bf Acknowledgments.}  The authors thank Jean-Louis
Colliot-Th\'el\`ene, Skip Garibaldi, Yong Hu, and Annette Maier for helpful comments on this manuscript.

%%%%%%%%%%%%%%%%%%%
%%%%%%%%%%%%%%%%%%%
\section{Patching and local-global principles for cohomology} \label{abstract patching}
%%%%%%%%%%%%%%%%%%%
%%%%%%%%%%%%%%%%%%%

This section considers patching and local-global principles for cohomology in an abstract algebraic setting, in which we are given a field and a finite collection of overfields indexed by a graph.  The results here will afterwards be applied to a geometric setting in Section~\ref{curve section}, where we will consider curves over a complete discretely valued field.

In the situation here, we will obtain a new long exact sequence 
for Galois cohomology with respect to the given field and its overfields, which 
in a key special case can be interpreted as a Mayer-Vietoris sequence.
In~\cite[Theorem~2.4]{HHK:H1}, we obtained such a sequence for linear algebraic groups that need not be commutative.  Due to the lack of commutativity, the assertion there was just for $H^0$ and $H^1$;
and that result was then used in~\cite{HHK:H1} to
obtain local-global principles for torsors in a more geometric context.  In the present paper, we consider commutative linear algebraic groups, and so higher cohomology groups $H^n$ are defined.  It is for these that we prove our long exact sequence, which we then use to obtain
a local-global principle for Galois cohomology in the key case of $H^n(F,\Z/m\Z(n-1))$ with $n > 1$.  This is carried out in Sections~\ref{abstract MV}
and~\ref{curve patches}.
(Note that the six-term cohomology sequence in~\cite[Theorem~2.4]{HHK:H1} is used in our arguments here, in the proofs of Theorems~\ref{loc-gl torsors} and~\ref{equiv factorization}.)

%%%%%%%%%%%%%%%%%%%
\subsection{$\Gamma$-Fields and patching} \label{gamma fields}
%%%%%%%%%%%%%%%%%%%

Our local-global principles will be obtained by an approach that
\textit{formally} emulates the notion of a cover of a topological
space by a collection of open sets, in the special case that there are
no nontrivial triple overlaps. In this case, one may ask to what
extent one may derive global information from local information with
respect to the sets in the open cover.  We encode this setup
combinatorically in the form of a graph whose vertices correspond to
the connected open sets in the cover and whose edges correspond to the
connected components of the overlaps (though we do not 
introduce an associated topological space or Grothendieck topology).

In our setting the global space will correspond to a field $F$ whose
arithmetic we would like to understand, and the open sets and overlaps
correspond to field extensions of $F$. This setup is formalized in the
definitions below, which draw on terminology in~\cite{HH} and~\cite{HHK:H1}.

%%%%%%%%%
\subsubsection{Graphs and $\Gamma$-fields}
%%%%%%%%%

By a \textit{graph} $\Gamma$, we will always mean a finite multigraph,
with a vertex set $\mc V$ and an edge set $\mc E$; i.e.\ we will
permit more than one edge to connect a pair of vertices.  But we will
not permit loops at a vertex: the two endpoints of an edge are
required to be distinct vertices.  

By an {\em orientation} on $\Gamma$ we will mean a choice of labeling
of the vertices of each edge $e \in \mc E$, with one chosen to be
called the {\em left vertex} $l(e)$ and the other the {\em right
vertex} $r(e)$ of $e$.  This choice can depend on the edge (i.e.\ a
vertex $v$ can be the right vertex for one edge at $v$, and the left
vertex for another edge at $v$).    

\begin{defn}
Let $\Gamma$ be a graph. A \textit{$\Gamma$-field} $F_\bullet$
consists of the following data:
\begin{enumerate}
\item
For each $v \in \mc V$, a field $F_v$,
\item
For each $e \in \mc E$, a field $F_e$,
\item
An injection $\iota_v^e: F_v \to F_e$ whenever $v$ is a vertex of the edge $e$.
\end{enumerate}
\end{defn}

Often we will
regard $\iota_v^e$ as an inclusion, and not write it
explicitly in the notation if the meaning is clear.

A $\Gamma$-field $F_\bullet$ can also be interpreted as an inverse system of fields.  Namely, the index set of the inverse system is the disjoint union $\mc V \sqcup \mc E$; 
and the maps consist of inclusions of fields 
$\iota_v^e: F_v \hookrightarrow F_e$ 
as above.

Conversely, consider any finite inverse system of fields whose index set can be partitioned into two subsets $\mc V \sqcup \mc E$, such that for each $e \in \mc E$ there are exactly two elements in $v,v' \in \mc V$ 
having maps $F_v \hookrightarrow F_e$ and $F_{v'} \hookrightarrow F_e$ in the inverse system; and such that there are no other maps in the inverse system.  Then such an inverse system of fields, called a {\em factorization inverse system} in \cite[Section~2]{HHK:H1}, gives rise to a graph $\Gamma$ and a $\Gamma$-field $F_\bullet$ as above.

Given a $\Gamma$-field $F_\bullet$, we may consider the inverse limit $F_{\Gamma}$ of the fields in $F_\bullet$, with respect to the associated inverse system, in the category of rings.  Equivalently, 
\[F_{\Gamma} = \left\{\left.a_\bullet \in \Fv
\ \right| \  \iota_v^e a_v =
\iota_w^e a_w \text{ for each $e$ incident to $v$ and $w$}\right\}.\]
We may also regard $F_{\Gamma}$ as a subring of $\Fe$, by sending an element $a_\bullet = (a_v)_{v \in \mc V}$ to $(a_e)_{e
\in \mc E}$, where $a_e = \iota_v^e
a_v = \iota_w^e a_w$ if $e$ is incident to $v$ and $w$.

Note that if $F_\bullet$ is a $\Gamma$-field, then we may regard each
field $F_v, F_e$ naturally as an $F_{\Gamma}$-algebra in such a way
that all the inclusions $\iota_v^e$ are $F_{\Gamma}$-algebra homomorphisms.  

\begin{lem}
If $F_\bullet$ is a $\Gamma$-field, then $F_{\Gamma}$ is a field if and
only if $\Gamma$ is connected. 
\end{lem}

\begin{proof}
If $\Gamma$ is disconnected, there are 
elements $a_\bullet$ of the inverse limit $F_{\Gamma}$
such that $a_\xi = 0$ for all $\xi \in \mc V \sqcup \mc E$ that lie on one connected
component of $\Gamma$, but $a_\xi = 1$ for all $\xi$ on another component.  
Hence 
$F_{\Gamma}$ has zero-divisors and is not a
field.  Conversely, if $F_{\Gamma}$ is not a field, then there is a zero-divisor $a_\bullet$.  The set of vertices and edges $\xi$ such that $a_\xi = 0$ forms an open subset of $\Gamma$, since $\iota_v^e
a_v = a_e = \iota_w^e a_w$ whenever $v,w$ are the vertices of an edge $e$.  This open subset is neither empty nor all of $\Gamma$, since $a_\bullet$ is a zero divisor.
Hence $\Gamma$ is disconnected.
\end{proof}

\begin{notn}
We will say for short that $F_\bullet$ is a \textit{$\Gamma/F$-field} if 
$\Gamma$ is a connected graph, $F$ is a field, 
and $F_\bullet$ is a $\Gamma$-field with $F_{\Gamma} = F$.
\end{notn}

%%%%%%%
\subsubsection{Patching problems}
%%%%%%%

Given a $\Gamma/F$-field $F_\bullet$, and a finite dimensional vector space $V$ over $F$, 
we obtain an inverse system $V_{F_\xi} = V \otimes_F F_\xi$ of finite dimensional vector spaces over the fields $F_\xi$ (for $\xi \in \mc V \sqcup \mc E$).  
Conversely, given such an inverse system, we can ask whether it is induced by an $F$-vector space $V$.  
More precisely, let $\Vect(F)$ be the category of finite dimensional $F$-vector spaces; define a \textit{vector space patching problem} $V_\bullet$ over $F_\bullet$ to be an inverse system of finite dimensional $F_\xi$-vector spaces; 
and let $\PP(F_\bullet)$ be the category of vector space patching problems over $F_\bullet$.  
There is then a base change functor $\Vect(F) \to \PP(F_\bullet)$.  
If it is an equivalence of categories, we say that \textit{patching holds for finite dimensional vector spaces} over the $\Gamma/F$-field
$F_\bullet$.  

We may consider analogous notions for other objects over $F$.  In particular
let $A$ be a group scheme over $F$ (which we always assume to be of finite type).  
Let $\Tors(A)$ denote the category of $A$-torsors over $F$; the objects in this category are classified by the elements in the Galois cohomology group $H^1(F,A)$.  

An object $\mc T$ in $\Tors(A)$ induces an {\em $A$-torsor patching problem} $\mc T_\bullet $ over $F_\bullet$, 
i.e.\ an inverse system consisting of $A_{F_\xi}$-torsors $\mc T_\xi$ for each $\xi \in \mc V \sqcup \mc E$,
together with isomorphisms $\phi_v^e:(\mc T_v)_{F_e} \to \mc T_e$ for $v$ a vertex of an edge $e$.  These patching problems form a category $\PP(F_\bullet,A)$, whose
morphisms correspond to collections of morphisms of torsors which
commute with the maps $\phi_v^e$.
(Once we choose an orientation on the graph $\Gamma$, an $A$-torsor patching problem can also be viewed as collection of $A$-torsors $\mc T_v$
for $v \in \mc V$, together with a choice of isomorphism $(\mc T_{l(e)})_{F_e} \to
(\mc T_{r(e)})_{F_e}$
for every edge $e \in \mc E$.)  
As before, we obtain a base change functor $\Tors(A) \to
\PP(F_\bullet,A)$; and we say that \textit{patching holds for
$A$-torsors} over the $\Gamma/F$-field $F_\bullet$ if this is an
equivalence of categories.  For short we say that \textit{patching
holds for torsors over $F_\bullet$} if it holds for all linear algebraic groups $A$ over $F$.
(Our convention is that a \textit{linear algebraic group} over $F$ is a 
smooth closed subgroups $A \subseteq \GL_{n,F}$ for some $n$.)

%%%%%%%%%
\subsubsection{Local-global principles and simultaneous factorization}
\label{lgpsf}
%%%%%%%%%

Local-global principles are complementary to patching.  Given a $\Gamma/F$-field $F_\bullet$, and a group scheme $A$ over $F$, we say that {\em  $A$-torsors over $F$ satisfy a local-global principle} over $F_\bullet$ if an $A$-torsor $\mc T$ is trivial if and only if each induced $F_v$-torsor ${\mc T}_v := \mc T \times_F F_v$ is trivial.  
In~\cite{HHK:H1}, criteria were given for patching and for local-global principles in terms of factorization.  Before recalling them, we introduce some terminology and notation.

If $F_\bullet$ is a $\Gamma/F$-field, and if $\Gamma$ is given an orientation, then there are induced maps $\pi_l,\pi_r:\Fv \to \Fe$ defined by 
$(\pi_l(a))_e = a_{l(e)}$ and $(\pi_r(a))_e = a_{r(e)}$ for $a = (a_v)_{v \in \mc V} \in \Fv$.  Similarly, if $A$ is a group scheme over $F$, there are induced maps $\pi_l,\pi_r:\prod_{v \in \mc V} A(F_v) \to \prod_{e \in \mc E} A(F_e)$
given by the same expressions, for $a = (a_v)_{v \in \mc V} \in \prod_{v \in \mc V} A(F_v)$.
We say that a group scheme $A$ over $F$ {\em satisfies simultaneous factorization} over a $\Gamma/F$-field $F_\bullet$ (or for short, is {\em factorizable} over $F_\bullet$) if the map of pointed sets $\pi_l \cdot \pi_r^{-1}:\prod_{v \in \mc V} A(F_v) \to \prod_{e \in \mc E} A(F_e)$, defined by $a \mapsto \pi_l(a)\pi_r(a)^{-1}$, is surjective. 
In other words, if we are given a collection of elements $a_e \in A(F_e)$ for all $e \in \mc E$, then there exist elements $a_v \in A(F_v)$ for all $v \in \mc V$ such that $a_e = a_{l(e)} a_{r(e)}^{-1}$ for all $e$, with respect to the inclusions $F_{l(e)},F_{r(e)} \hookrightarrow F_e$.  
Note that this factorization condition does not depend on the choice of orientation, since if we reverse the orientation on an edge $e$ then we may consider the element $a'\in\prod_{e \in \mc E} A(F_e)$ such that $a'_e = a_e^{-1}$ and where the other entries of $a'$ are the same as for $a$.  

%%%%%%%%%
\subsubsection{Relations between patching, local-global principles
and factorization}
%%%%%%%%%

The following two results are essentially in \cite{HHK:H1}.

\begin{thm} \label{equiv patching}
Let $\Gamma$ be a connected graph, $F$ a field, and $F_\bullet$ a $\Gamma/F$-field.  Then the following conditions are equivalent:
\begin{enumerate}
\renewcommand{\theenumi}{\roman{enumi}}
\renewcommand{\labelenumi}{(\roman{enumi})}
\item \label{sim fact}
$\GL_n$ is factorizable over $F_\bullet$ for all $n \ge 1$.
\item \label{vs patching}
Patching holds for finite dimensional vector spaces over $F_\bullet$.
\item \label{torsor patching}
Patching holds for torsors over $F_\bullet$.
\end{enumerate}
\end{thm}

\begin{proof}
It was shown in \cite[Proposition~2.2]{HHK:H1} that~(\ref{sim fact}) is equivalent to~(\ref{vs patching}); and in 
\cite[Theorem~2.3]{HHK:H1} it was shown that~(\ref{vs patching}) implies~(\ref{torsor patching}).  It remains to show that~(\ref{torsor patching}) implies~(\ref
{sim fact}).  

Fix an orientation for $\Gamma$ and let $g = (g_e)_{e \in \mc E}\in \GL_n(\Fe)$.  
We wish to show that there exists
$h \in \GL_n(\Fv)$
such that $g = \pi_l(h) \pi_r(h)^{-1}$. 

Consider the patching problem 
for $\GL_n$-torsors over $F_\bullet$ that is given by trivial torsors 
$\mc T_\xi := \GL_{n,F_\xi}$
over $F_\xi$ for each $\xi \in \mc V \sqcup \mc E$, together with transition functions $\GL_{n,F_e} = (\mc T_{r(e)})_{F_e} \to
(\mc T_{l(e)})_{F_e} = \GL_{n,F_e}$ given by multiplication by $g_e \in \GL_n(F_e)$, for each $e \in \mc E$.
By hypothesis~(\ref{torsor patching}), there is a $\GL_n$-torsor $\mc T$ over $F$ that induces this patching problem.  
But $\mc T$ is trivial, since $H^1(F,\GL_n)=0$ by Hilbert's Theorem~90 (\cite[Theorem~29.2]{BofInv}); i.e.\ there is an isomorphism $\GL_{n,F} \to \mc T$.  The induced isomorphisms $\GL_{n,F_v} \to \mc T_v = \GL_{n,F_v}$ 
are given by multiplication by elements $h_v \in \GL_n(F_v)$.  Since $\mc T$ induces the given patching problem, 
we have $h_{l(e)} = g_e h_{r(e)} \in \GL_n(F_e)$ for every $e \in \mc E$.
Therefore $g = \pi_l (h) \pi_r (h)^{-1}$, with $h = (h_v)_{v \in \mc V} \in \GL_n(\Fe)$, as desired. 
\end{proof} 

\begin{thm} \label{loc-gl torsors}
Let $\Gamma$ be a connected graph, $F$ a field, and $F_\bullet$ a $\Gamma/F$-field.  Assume that patching holds for finite dimensional vector spaces over $F_\bullet$.  
Then a linear algebraic group 
$A$  over $F$ is factorizable over $F_\bullet$ if and only if 
$A$-torsors over $F$ satisfy a local-global principle over $F_\bullet$.
\end{thm}

\begin{proof}
This assertion is contained in the exactness of the sequence given in 
\cite[Theorem~2.4]{HHK:H1}; i.e.\ $\pi_l \cdot \pi_r^{-1}$ is surjective if and only if $H^1(F,A) \to \prod_{v \in \mc V} H^1(F_v,A)$ is injective.
\end{proof}

Note that the hypothesis of Theorem~\ref{loc-gl torsors} does not
imply that the equivalent conditions in the conclusion of that theorem
necessarily hold.  (In particular, in Example~4.4 of~\cite{HHK} there
is a non-trivial obstruction to a local-global principle, by
Corollaries~5.6 and~5.5 of~\cite{HHK:H1}).  Thus patching need not
imply factorization over $F_\bullet$ for all linear algebraic groups
over $F$.  But as shown in the next section (Corollary~\ref{patching
yields sep fact}), patching does imply factorization for all linear
algebraic groups if we are allowed to pass to the separable closure of
$F$.  This will be useful in obtaining local-global principles for
higher cohomology.

%%%%%%%%%%%%%%
\subsection{Separable factorization} \label{sep fact sect}
%%%%%%%%%%%%%%

As asserted in Theorems~\ref{equiv patching} and~\ref{loc-gl torsors}, there are relationships between factorization conditions on the one hand, and patching and local-global properties on the other.  Below, in 
Theorem~\ref{equiv factorization} and Corollary~\ref{patching yields sep fact}, we prove related results of this type, concerning ``separable factorization'', which will be needed later in applying the results of Section~\ref{abstract MV}.  We also prove a result (Proposition~\ref{sep fact ex seq}) that will be used in obtaining our long exact sequence in Section~\ref{abstract MV}, and hence our local-global principle there.

%%%%%%%%%
\subsubsection{Separably factorizable group schemes} \label{sep fact subsub}
%%%%%%%%%

Let $F_\bullet$ be a $\Gamma/F$-field, and write $\Gal(F)$ for the absolute
Galois group $\Gal(F\s/F)$.
Given an $F$-scheme $A$, we have morphisms
$A(F\s) \to A(F_v \otimes_F F\s)$ for each
vertex $v \in \mc V$, and 
$A(F_v \otimes_F F\s) \to A(F_e \otimes_F F\s)$ when $v \in \mc V$ is a vertex of $\Gamma$ on the edge $e \in \mc E$.  
These are induced by the inclusions $F\s \to F_v \otimes_F
F\s$ and $F_v \otimes_F F\s \to F_e \otimes_F F\s$.

If $A$ is an $F$-scheme, and $L \subseteq L'$ are field extensions of $F$, then the natural map $A(L) \to A(L')$ is an inclusion.  (This is immediate if $A$ is affine, and then follows in general.)  In particular, given
a $\Gamma/F$-field $F_\bullet$ as above,
the maps $A(F) \to A(F_v)$ and $A(F_v) \to A(F_e)$ are injective
for $v$ a vertex of an edge $e$ in $\Gamma$. 

If we choose an orientation for the graph $\Gamma$, then as in
Section~\ref{lgpsf} we may define maps $\pi_l,\pi_r: \Av \to \Ae$ by 
$(\pi_l(a))_e = a_{l(e)}$ and $(\pi_r(a))_e = a_{r(e)}$.

\begin{lem} \label{affine equalizer}
Consider an affine $F$-scheme $A$, a graph $\Gamma$ with a choice of orientation, and a 
$\Gamma/F$-field $F_\bullet$.
Then 
\[
\xymatrix{
A(F^{sep}) \ar[r] & \Av \ar@<-.7ex>[r]_{\pi_r} \ar@<+.7ex>[r]^{\pi_l} &
\Ae
}
\]
is an equalizer diagram of sets.
\end{lem}

\begin{proof}
The hypothesis that $F$ equals $F_{\Gamma}$ is equivalent to
having an exact sequence of $F$-vector spaces
\(0 \to F \to \Fv \to \Fe\), given by 
$\pi_l \cdot \pi_r^{-1}$
on the
right. 
Since $F\s$ is a flat $F$-module,
we have an exact sequence
\(0 \to F\s \to \Fv \otimes F\s \to \Fe \otimes F\s\).
This in turn tells us that in the category of rings,
\[\displaystyle F\s = \lim_{\underset{\xi \in \mc V \sqcup
\mc E}{\longleftarrow}} (F_\xi \otimes F\s).\]
Write $A = \Spec(R)$. By the inverse limit property above, it
follows that a homomorphism $R \to F\s$ is equivalent to a homomorphism
$\phi: R \to \Fv \otimes F\s$
such that $\pi_l \phi = \pi_r \phi$, where $\pi_l, \pi_r : \Fv
\otimes F\s
\to \Fe \otimes F\s$ are the two projections.  This gives the desired equalizer diagram.
\end{proof}

This lemma, and the notion of factorizability in the previous section, motivate the following definition.  

\begin{defn} \label{factorizable definition}
Let $F_\bullet$ be a $\Gamma/F$-field, and suppose that $A$ is a group scheme over
$F$. We say that $A$ is \textit{separably factorizable} (over $F_\bullet$)
if the pointed set map $\pi_l \cdot \pi_r^{-1}: \Av \to \Ae$ 
is surjective
for some (hence every) orientation on $\Gamma$.
\end{defn}

Lemma~\ref{affine equalizer} and Definition~\ref{factorizable definition} then yield:

\begin{prop} \label{sep fact ex seq}
Let $F_\bullet$ be a $\Gamma/F$-field, and let $A$ be a group scheme over $F$.  Choose any orientation on $\Gamma$, and take the associated maps $\pi_l, \pi_r$.  Then 
$A$ is   
separably factorizable if and only if 
\[
\xymatrix @C=1.5cm{
0 \ar[r] & A(F\s) \ar[r] & \Av \ar[r]^{\pi_l \cdot \pi_r^{-1}} &  \Ae \ar[r] & 0
}
\]
is an exact sequence of pointed $\Gal(F)$-sets (and in fact an exact sequence of Galois modules in the case that $A$ is commutative).
\end{prop}

%%%%%%%%%
\subsubsection{Patching and separable factorization}
%%%%%%%%%

The following theorem and its corollary complement Theorems~\ref{equiv patching} and~\ref{loc-gl torsors}.

\begin{thm} \label{equiv factorization}
Let $\Gamma$ be a connected graph, $F$ a field, and $F_\bullet$ a $\Gamma/F$-field.  Then the following conditions are equivalent:
\begin{enumerate}
\renewcommand{\theenumi}{\roman{enumi}}
\renewcommand{\labelenumi}{(\roman{enumi})}
\item \label{fact GL}
$\GL_n$ is factorizable over $F_\bullet$, for all $n\ge 1$.
\item \label{sep fact GL}
$\GL_n$ is separably factorizable over $F_\bullet$, for all $n\ge 1$.
\item \label{sep fact all}
Every linear algebraic group over $F$ is separably factorizable over $F_\bullet$.
\end{enumerate}
\end{thm}

\begin{proof}
We begin by showing that~(\ref{fact GL}) implies~(\ref{sep fact all}).  Fix an orientation for $\Gamma$.
Let $A$ be a 
linear algebraic group over $F$, 
and suppose we are given $g \in \Ae$. 
We wish to show that there exists
$h \in \Av$
such that $g = \pi_l(h) \pi_r(h)^{-1}$. 

Since $\mc E$ is finite, there is a finite separable field extension $L/F$ such
that $g$ is the image of $g' \in \prod_{e \in \mc E} A(F_e \otimes_F L)$. 
Let $A' = R_{L/F}A_L$, the Weil restriction of $A_L = A \times_F L$
from $L$ to $F$ (see \cite{BLR:Neron}, Section~7.6); this is a linear algebraic group over $F$.  We may then view $g' \in \prod_{e \in \mc E} A'(F_e)$.

Since $\GL_n$ is factorizable over $F_\bullet$ by condition~(\ref{fact GL}), Theorem~\ref{equiv patching} implies that patching holds for finite-dimensional vector spaces over $F_\bullet$.  Thus~\cite[Theorem~2.4]{HHK:H1} applies, 
giving us a six-term cohomology sequence for $A'$; 
and we may consider the image of $g$ under 
the coboundary map $\prod_{e \in \mc E} A'(F_e) \to
H^1(F,A')$.  This image defines an $A'$-torsor $\mc T'$ over $F$ (viz.\ the solution to the patching problem that consists of trivial torsors over each $F_v$ and for which the transition functions are given by $g'$).  But $H^1(F,A')$ may be identified with $H^1(L,A)$ by Shapiro's Lemma~(\cite{Se:GC}, Corollary to Proposition~I.2.5.10), since 
$A'(F\s)$ is the Galois module induced from $A(F\s)$ via the inclusion $\Gal(L) \to \Gal(F)$.  So $\mc T'$ corresponds to an $A$-torsor $\mc T$ over $L$.  
There is then a finite separable field
extension $E/L$ 
over which $\mc T$ becomes trivial.
After replacing $L$ by $E$, we may assume that $\mc T$ and hence $\mc T'$ is trivial.  Hence by the exactness of the six-term sequence in~\cite[Theorem~2.4]{HHK:H1}, $g'$ is the image of an element 
$h' \in \prod_{v \in \mc V} A'(F_v) = \prod_{v \in \mc V} A(F_v \otimes_F L)$ 
under $\pi_l \cdot \pi_r^{-1}$. The image $h \in \Av$ of $h'$ is then as desired, proving that condition~(\ref{sep fact all}) holds.

Condition~(\ref{sep fact all}) trivially implies condition~(\ref{sep fact GL}).  It remains to show that condition~(\ref{sep fact GL}) implies condition~(\ref{fact GL}).

If condition~(\ref{sep fact GL}) holds, then Proposition~\ref{sep fact ex seq}
yields a short exact sequence of
pointed $\Gal(F)$-sets
\[
\xymatrix 
@C=1cm
{
0 \ar[r] & \GL_n(F\s) \ar[r] & {{\prod_{v \in \mc V} \GL_n(F_v \otimes_F F\s)}}
\ar[r]^{\pi_l \cdot \pi_r^{-1}} &  {{\prod_{e \in \mc E} \GL_n(F_e \otimes_F F\s)}} \ar[r] & 0.
}
\]
This in turn yields an exact sequence of pointed sets in Galois cohomology that begins
\[0 \to H^0(F,\GL_n) \to \prod_{v \in \mc V} H^0(F,(\GL_n)_v)
\to \prod_{e \in \mc E} H^0(F,(\GL_n)_e) \to H^1(F,\GL_n).\]
But the last term vanishes by Hilbert's Theorem~90.  The remaining short exact sequence is then equivalent to the condition that $\GL_n$ is factorizable over $F_\bullet$, i.e.\ condition~(\ref{fact GL}).  
\end{proof}

\begin{cor} \label{patching yields sep fact}
Let $\Gamma$ be a connected graph, $F$ a field, and $F_\bullet$ a
$\Gamma/F$-field.  Then patching holds for finite dimensional vector
spaces over $F_\bullet$ if and only if every linear algebraic group over $F$ is separably factorizable over $F_\bullet$.
\end{cor}

\begin{proof}
This is immediate from Theorem~\ref{equiv patching} and 
Theorem~\ref{equiv factorization}, which assert that these two conditions are each equivalent to $\GL_n$ being factorizable over $F_\bullet$ for all $n \ge 1$.
\end{proof}

%%%%%%%%%%%%%%%%%%%
\subsection{Globally dominated field extensions and cohomology} \label{gl dom sect}
%%%%%%%%%%%%%%%%%%%

To obtain our results, we will want to relate the cohomologies $H^n(F,
A)$ and $H^n(F_\xi, A)$ for $\xi$ a vertex or edge of $\Gamma$.  One
difficulty with this in general is the potential difference between
the absolute Galois groups of $F$ and $F_\xi$.  To bridge this gap, 
we will use the Galois module $A(F_\xi \otimes_F F\s)$, which was studied in
Section~\ref{sep fact subsub}.  Its cohomology $H^n(\Gal(F),A(F_\xi \otimes_F F\s))$ is meant to approximate the cohomology group
$H^n(F_\xi, A)$.  

Our strategy will be carried out using the 
notion of ``global domination,'' which we introduce and study below.
The condition that a Galois module has globally dominated cohomology
will provide an important ingredient in demonstrating the existence of
Mayer-Vietoris type sequences and local-global principles for its
Galois cohomology groups. These applications are developed in
Section~\ref{abstract MV}.

%%%%%%%%%
\subsubsection{Globally dominated extensions}
%%%%%%%%%

\begin{defn}
Fix a field $F$.  For any field extension $L/F$, with separable closure $L\s$, let $L\gd$ denote the compositum of $L$ and $F\s$
taken within $L\s$.  If $E/L$ is a separable algebraic field extension, we say that $E/L$ is \textit{globally dominated} (with respect to $F$) if $E$ is contained in $L\gd$.
\end{defn}

Thus a separable algebraic field extension $E/L$ is globally dominated if and only if $E$ is
contained in some compositum $E' L \subseteq L\s$, where $E'/F$ is a separable
algebraic field extension.  Also, the subfield $L\gd \subseteq L\s$ can be characterized as the maximal globally dominated field extension of $L$.
Since the extension $F\s/F$ is Galois with group $\Gal(F)$, it follows that the extension $L\gd/L$ is Galois and that
$\Gal\gd(L) := \Gal(L\gd/L)$ can be identified with a subgroup of $\Gal(F)$.

\begin{lem} \label{global cohomology lem}
Let $L/F$ be a field extension, and let $A$ be a commutative group
scheme defined over $F$. Then we may identify:
\[H^n(\Gal(F), A(L \otimes_F F\s)) = 
H^n(\Gal\gd(L), A(L\gd)).\]
\end{lem}

\begin{proof}
We may identify the group 
\(H^n(\Gal(F), A(L \otimes F\s)\))
as a limit of groups 
\[H^n(\Gal(E/F), A(L \otimes E)),\]
and the group 
\(H^n(\Gal\gd(L), A(L\gd))\) 
as a limit of groups 
\[H^n(\Gal(LE/L), A(LE)),\]
where both limits are taken
over finite Galois extensions $E/F$, and where $LE$ is a compositum of $L$ and
$E$. Therefore the result will follow from a (compatible) set of
isomorphisms
\[H^n(\Gal(E/F), A(L \otimes E)) \cong H^n(\Gal(LE/L),
A(LE)).\]

Write $L \otimes E = \prod_{i = 1}^m E_i$ for finite Galois field extensions
$E_i/L$. We can also choose $LE = E_1$. We have $A(L \otimes E)
= \prod_i A(E_i)$. Let $G = \Gal(E/F)$ and let
$G_1$ be the stabilizer of $E_1$ (as a set) with respect to the action of $G$ on
$L \otimes_F E$. Then we may identify the $G$-modules $A(L \otimes E)$ and
$\Ind_{G_1}^G A(E_1)$. We therefore have
\[H^n(G, A(L \otimes E)) \cong 
H^n(G, \Ind_{G_1}^G A(E_1)) \cong 
H^n(G_1, A(E_1)) =
H^n(\Gal(LE/L), A(LE))\]
by Shapiro's Lemma (\cite{Se:GC}, Corollary to Proposition~I.2.5.10), as desired.
\end{proof}

%%%%%%%%%
\subsubsection{Globally dominated cohomology}
%%%%%%%%%

It remains to compare the cohomology
with respect to the maximal globally dominated extension and the full
Galois cohomology.  For this we make the following

\begin{defn}
Let $A$ be a commutative group scheme over $F$ and $L/F$ a field
extension. We say that \textit{the cohomology of $A$ over $L$ is
globally dominated} (with respect to $F$) if $H^n(L\gd,A)=0$ for every $n>0$.
\end{defn}

\begin{prop} \label{refinement}
Let $A$ be a commutative group scheme over $F$ and $L/F$ a field
extension. Suppose that the cohomology of $A$ over $L$ is globally
dominated. Then we have isomorphisms:
\[H^n(\Gal(F), A(L \otimes F\s)) = H^n(\Gal\gd(L), A(L\gd)) =
H^n(\Gal(L), A(L\s))\]
for all $n \geq 0$.
\end{prop}

\begin{proof}
The identification of the first and second groups was given in Lemma~\ref{global cohomology lem}, and it remains to prove the isomorphism between the second and third groups.  By the global domination hypothesis, $H^n(\Gal(L\gd), A(L\s)) = H^n(L\gd,A)= 0$ for all $n>0$.  Hence from
the Hochschild-Serre
spectral sequence
\[H^p(\Gal\gd(L), H^q(\Gal(L\gd), A(L\s))) \implies
H^{p + q}(\Gal(L), A(L\s))\]
for the tower of field extensions $L \subseteq L\gd
\subseteq L\s$
(viz.\ by \cite[Theorem~III.2]{HS}),
the desired isomorphism follows.
\end{proof}

The notion of globally dominated cohomology can also be described just in terms of finite extensions of fields.  First we prove a lemma.

\begin{lem} \label{coho limit}
Suppose that a field $E_0$ is a filtered direct limit of subfields $E_i$, each of which is an extension of a field $E$.  Let $A$ be a commutative group scheme over $E$, and let $\alpha \in H^n(E,A)$ for some $n \ge 0$.  If the induced element $\alpha_{E_0} \in H^n(E_0,A)$ is trivial, then there is some $i$ such that $\alpha_{E_i} \in H^n(E_i,A)$ is trivial.
\end{lem}

\begin{proof} 
Since $\alpha_{E_0} \in H^n(E_0,A)$ is trivial, we may find some finite Galois
extension $L/E_0$ such that $\alpha_{E_0}$ may be written as a cocycle in 
$Z^n(L/E_0, A(L))$ and such
that it is the coboundary of a cochain in $C^{n-1}(L/E_0,
A(L))$.  Now the Galois extension $L/E_0$ is generated by finitely
many elements of $L$, and the splitting cochain is defined by
an additional finite collection of elements in $A(L)$, each of which is defined over some finitely generated extension of $E$ (since $A$ is of finite type over $E$).  So we may find finitely
many elements $a_1, \ldots, a_r \in E_0$ such that $\alpha_{E(a_1, \ldots,
a_r)} = 0$. But since $E_0$ is the filtered direct limit of the fields $E_i$, there is an $i$ such that $a_1, \ldots, a_r \in E_i$; and then
$\alpha_{E_i} = 0$ as desired.
\end{proof}

\begin{prop} \label{gl dom cond}
Let $A$ be a commutative group scheme over $F$ and $L/F$ a field
extension.   Then the cohomology of $A$ over $L$ is
globally dominated if and only if for every finite globally dominated field extension
$L'/L$, every $n>0$, and every $\alpha \in H^n(L', A)$, there exists a
finite globally dominated extension $E/L'$ such that $\alpha_E = 0$.
\end{prop}

\begin{proof}
First suppose that the cohomology of $A$ over $L$ is
globally dominated, and let $\alpha \in H^n(L', A)$ for some finite globally dominated field extension
$L'/L$ and some $n>0$.
Then $\alpha_{L\gd}=0$ by hypothesis; and so by Lemma~\ref{coho limit} there is some finite globally dominated extension $E/L'$ such that $\alpha_E = 0$, as desired.

Conversely, suppose that the above condition on every $\alpha \in H^n(L', A)$ holds.  Let $\alpha \in H^n(L\gd,A)$.  Then $\alpha$ is in the image of some element $\til\alpha \in H^n(L',A)$ for some finite extension $L'/L$ that is contained in $L\gd$.  Now $L'$ is globally dominated, so by hypothesis there exists a finite globally dominated field extension $E/L'$ such that $\til \alpha_E=0$.  Thus $\alpha = \til\alpha_{L\gd} = (\til \alpha_E)_{L\gd} = 0$.  This shows that $H^n(L\gd,A)$ is trivial, so the cohomology is globally dominated.
\end{proof}

%%%%%%%%%%%%%%%%%%%
\subsection{Criteria for global domination} \label{gl dom cond sect}
%%%%%%%%%%%%%%%%%%%

In the case of cyclic groups, the condition for cohomology to be
globally dominated will be made more explicit here, using the Bloch-Kato conjecture to reduce to consideration of just the
first cohomology group.  That conjecture asserts the surjectivity of the norm residue homomorphism $K_\bullet^{\rm M}(F) \to H^\bullet(F, \zmod m {\bullet})$ 
of graded rings.
This surjectivity was proven for $m$ prime in \cite{Voe:MCl} and \cite{Wei}, with the general case then following via \cite{GS}, Proposition~7.5.9.  Since every element in 
the Milnor K-group
$K_n^{\rm M}(F)$ is by definition a sum of $n$-fold products of elements of $K_1^{\rm M}(F)$, the following assertion is then immediate: 

\begin{prop} \label{cup products}
Let $F$ be a field and let $m$ be a positive integer not divisible by $\cha(F)$.  Then for every $n \ge 1$, every element of $H^n(F,
\zmod m {n})$ is a sum of $n$-fold cup products of elements of $H^1(F, \zmod m {1})$.
\end{prop} 

%%%%%%%%%
\subsubsection{Global domination for cyclic groups}
%%%%%%%%%

\begin{prop} \label{Kummer globally dominated}
Let $L/F$ be a field extension, and $m$ an integer not divisible by the
characteristic of $F$.  Then the following are equivalent:
\begin{enumerate}
\renewcommand{\theenumi}{\roman{enumi}}
\renewcommand{\labelenumi}{(\roman{enumi})}
\item \label{coho gl dom cond}
The cohomology of $\Z/m
\Z$ over $L$ is globally dominated.
\item \label{field gl dom cond}
For every finite globally dominated field extension $L'/L$ 
and every positive integer
$r$ dividing $m$, every $\Z/r
\Z$-Galois field extension of $L'$ is globally dominated.
\item \label{power Weier cond}
The multiplicative group $(L\gd)^\times$ is $m$-divisible; i.e.\ $((L\gd)^\times)^m = (L\gd)^\times$.
\end{enumerate}
\end{prop}

\begin{proof}
(\ref{coho gl dom cond}) $\Rightarrow$ (\ref{field gl dom cond}): 
A $\Z/r
\Z$-Galois field extension of $L'$ corresponds to an element $
\alpha \in H^1(L',\Z/r \Z(1))$.  Let $\beta$ be 
the image of $\alpha$ in $H^1(L\gd,\Z/r \Z(1)) 
= H^1(L\gd,\Z/r\Z)$, where the equality holds
because the field $L\gd=LF\s$ contains a primitive $m$-th root of unity.  
It suffices to show that $\beta=0$.

In the long exact cohomology sequence associated to the short exact sequence of constant groups $0 \to \Z/r \Z \to \Z/m \Z \to \Z/(m/r) \Z \to 0$, the map $H^0(L\gd,\Z/m \Z) \to H^0(L\gd,\Z/(m/r) \Z)$ is surjective, so the map $H^1(L\gd,\Z/r
\Z) \to H^1(L\gd,\Z/m
\Z)$ is injective.  But the latter group is trivial, by hypothesis.  Hence $\beta=0$.

(\ref{field gl dom cond}) $\Rightarrow$ (\ref{power Weier cond}):
Given $a \in (L\gd)^\times$, we wish to show that $a \in ((L\gd)^\times)^m$.  Let $\zeta$ be a primitive $m$-th root of unity in $F\s \subseteq L\gd$, and let $L' = L(\zeta,a) \subseteq L\gd$.  Thus $L'/L$ is finite and separable.  
The field $E = L'(a^{1/m}) \subseteq L\s$ is Galois over $L'$, with Galois group cyclic of order $r$ for some $r$ dividing $m$.  
Thus the extension $E/L'$ is globally dominated, by (\ref{field gl dom cond}); i.e.\ $E \subseteq {L'}\gd = L\gd$.  
Hence $a \in ((L\gd)^\times)^m$.

(\ref{power Weier cond}) $\Rightarrow$ (\ref{coho gl dom cond}): 
By (\ref{power Weier cond}), $H^1 (L\gd,\Z/m
\Z) = H^1 (L\gd,\zmod m 1) = (L\gd)^\times/((L\gd)^\times)^m$ is trivial.

It remains to show the triviality of $H^n(L\gd,\Z/m
\Z) = H^n(L\gd,\zmod m n)$ for $n>1$.  So let $\alpha$ lie in this cohomology group. 
By Proposition~\ref{cup products},
we may
write $\alpha = \sum_{i = 1}^s \alpha_i$, with each $\alpha_i$ having the form
\[\alpha_i = \alpha_{i,1} \cup \cdots \cup \alpha_{i,n},\]
where $\alpha_{i,j} \in H^1(L\gd, \Z/m \Z)$.  But this $H^1$ is trivial.  Hence each $\alpha_i$ is trivial, and $\alpha$ is trivial, as desired.
\end{proof}

%%%%%%%%%
\subsubsection{Global domination for commutative group schemes}
%%%%%%%%%

Using the above result, 
the question of global domination for the cohomology of a finite commutative group
scheme $A$ can be reduced to the case of cyclic groups of prime order.
We restrict to the case that the characteristic of $F$ does not divide the order of $A$ (equivalently, $A_{F\s}$ is a finite constant group scheme of order not divisible by $\cha(F)$).

\begin{cor} \label{cyclic global}
Let $L/F$ be a field extension, and $S$ a 
collection of prime numbers unequal to $\cha(F)$.  Suppose that the cohomology of the finite
constant group scheme $\Z/\ell \Z$ over $L$ is globally
dominated for each $\ell \in S$. Then for every finite commutative group
scheme $A$ over $F$ of order divisible only by primes in $S$,
the cohomology of $A$ over $L$ is globally dominated.
\end{cor}

\begin{proof}
We wish to show that $H^n(L\gd,A)=0$ for $n>0$.  Since $A$ is a finite \'etale group scheme defined over $F$, it becomes split (i.e.\ a finite constant group scheme) over $F\s$ and hence over $L\gd=LF\s$.  In particular, the base change of $A$ to $L\gd$ is a product of copies of cyclic groups $\mbb
Z/m \Z$, where each prime dividing $m$ lies in $S$.  Since cohomology commutes with taking products of coefficient groups, we are reduced to the case that
$A\cong \Z/m \Z$ for $m$ as above.  The result now follows from condition~(\ref{power Weier cond}) of Proposition~\ref{Kummer globally dominated}, since a group is $m$-divisible if it is $\ell$-divisible for each prime factor~$\ell$ of~$m$.
\end{proof}

In characteristic zero, we also obtain a result in the case of group schemes that need not be finite.  First we prove a lemma.  If 
$A$ is a group scheme over a field $E$ and 
$m \ge 1$, let $A[m]$ denote the $m$-torsion subgroup of $A$, i.e.\ the kernel of the map $A \to A$ given by multiplication by $m$.  Thus there is a natural map $H^n(E,A[m]) \to H^n(E,A)$.

\begin{lem} \label{torsion coho surjects}
Let $A$ be a connected commutative group scheme over a field $E$ of characteristic zero, and let $n \ge 1$.  Then every element of $H^n(E,A)$ is in the image of $H^n(E,A[m]) \to H^n(E,A)$ for some $m \ge 1$.
\end{lem}

\begin{proof}
The group $H^n(E,A)$ is torsion by \cite[I.2.2 Cor.~3]{Se:GC}, and so 
for every $\alpha \in H^n(E,A)$
there exists $m \ge 1$ such that $m\alpha=0$ (writing $A$ additively).  
Since $\cha(E)=0$ and $A$ is connected, there is a 
short exact sequence $0 \to A[m] \to A \to A \to 0$ of \'etale sheaves.  This yields an exact sequence 
$H^n(E,A[m]) \to H^n(E,A) \to H^n(E,A)$ of groups, where the latter map is multiplication by $m$.  
Thus $\alpha$ is sent to zero under this map, and hence it lies in the image of $H^n(E,A[m])$.  
\end{proof}

\begin{prop} \label{char 0 global}
Assume that $\cha(F)=0$, and let $L/F$ be a field extension.  Suppose that the cohomology of the finite
constant group scheme $\Z/\ell \Z$ over $L$ is globally
dominated for every prime $\ell$. Then for every smooth commutative group scheme $A$ over $F$, the cohomology of $A$ over $L$ is globally dominated.
\end{prop}

\begin{proof}
Let $\alpha \in H^n(L\gd,A)$, for some $n > 0$.  We wish to show that $\alpha=0$.  

The short exact sequence $0 \to A^0 \to A \to A/A^0 \to 0$ of \'etale sheaves 
yields an exact sequence 
$H^n(L\gd,A^0) \to H^n(L\gd,A) \to H^n(L\gd,A/A^0)$ of groups.  
But 
Corollary~\ref{cyclic global} asserts that the cohomology of the finite commutative group scheme $A/A^0$ over $L$ is globally dominated, since $\cha(F)=0$; 
i.e.\ $H^n(L\gd,A/A^0)=0$.  So $\alpha \in H^n(L\gd,A)$ is the image
of some element $\alpha^0 \in H^n(L\gd,A^0)$.
By Lemma~\ref{torsion coho surjects}, $\alpha^0$ 
lies in the image of $H^n(L\gd,A^0[m])$ for some $m \ge 1$.  
Since $A^0[m]$ is a finite commutative group scheme over $L$, a second application
of Corollary~\ref{cyclic global} yields that $H^n(L\gd,A^0[m])=0$.
So $\alpha^0 = 0$ and hence $\alpha=0$.
\end{proof}

%%%%%%%%%%%%%%%%%%%
\subsection{Mayer-Vietoris and local-global principles} \label{abstract MV}
%%%%%%%%%%%%%%%%%%%

We now use the previous results to obtain our long exact sequence, which in particular gives the abstract form of our Mayer-Vietoris sequence, and we then prove the abstract form of a local-global principle for Galois cohomology.

\begin{thm} \label{long ex sequence} 
Given an oriented graph $\Gamma$, fix a $\Gamma/F$-field $F_\bullet$ and consider a separably factorizable smooth commutative
group scheme $A$ over $F$.
Suppose that for every $\xi \in \mc V \sqcup \mc E$, 
the cohomology of $A$ over $F_\xi$ is globally dominated.
Then we have a long exact sequence of Galois
cohomology:
\[
\xymatrix{
0 \ar[r] & H^0(F,A) \ar[r] & \prod_{v \in \mc V} H^0(F_v,A) \ar[r] & \prod_{e
\in \mc E} H^0(F_e,A) 
\ar `d[l] `[lld] [lld] \\
& H^1(F, A) \ar[r] & \prod_{v \in \mc V} H^1(F_v, A) \ar[r] & \prod_{e
\in \mc E} H^1(F_e, A)\ar[r] & \cdots
}
\]
\end{thm}

\begin{proof}
By hypothesis, the cohomology of $A$ over $F_\xi$ is globally dominated. By
Proposition~\ref{refinement}, with $L = F_\xi$, we may identify
\(H^n(\Gal(F), A(F_\xi \otimes_F F\s))
\cong
H^n(F_\xi, A).\)

Since $A$ is separably factorizable, 
by Proposition~\ref{sep fact ex seq}
we have a short exact sequence of
$\Gal(F)$-modules
\[
\xymatrix 
@C=1cm
{
0 \ar[r] & A(F\s) \ar[r] & \Av \ar[r]^{\pi_l \cdot \pi_r^{-1}} &  \Ae \ar[r] & 0.
}
\]
This induces a long exact sequence in Galois cohomology over $F$.  Applying the 
above identification to the terms of this sequence, we obtain the exact sequence asserted in the theorem. 
\end{proof}

\begin{cor} \label{les conditions}
Given a separably factorizable smooth commutative
group scheme $A$ over $F$ and a $\Gamma/F$-field $F_\bullet$,
the long exact sequence in Theorem~\ref{long ex sequence} holds in each of the following cases:
\begin{enumerate}
\renewcommand{\theenumi}{\roman{enumi}}
\renewcommand{\labelenumi}{(\roman{enumi})}
\item \label{mv finite}
 $A$ is finite; and for every $\xi \in \mc V \sqcup \mc E$, and every prime $\ell$
dividing the order of $A$, the cohomology of $\Z/\ell
\Z$ over $F_\xi$ is globally dominated.
\item \label{mv char 0}
$F$ is a field of characteristic zero; and for every $\xi \in \mc V \sqcup \mc E$, and every prime number $\ell$, the cohomology of $\Z/\ell
\Z$ over $F_\xi$ is globally dominated.
\end{enumerate}
\end{cor}

\begin{proof}
By Theorem~\ref{long ex sequence} it suffices to show that 
the cohomology of $A$ over $F_\xi$ is globally dominated.  In these two cases, this condition is satisfied by
Corollary~\ref{cyclic global} and Proposition~\ref{char 0 global} respectively.
\end{proof}

An important case is that of a graph $\Gamma$ that is {\em bipartite}, i.e.\ for which there is a partition $\mc V = \mc V_0 \sqcup \mc V_1$ such that for every edge $e \in \mc E$, one vertex is in $\mc V_0$ and the other is in $\mc V_1$.  Given a bipartite graph $\Gamma$ together with such a partition, we will choose the orientation on $\Gamma$ given by taking $l(e)$ and $r(e)$ to be the vertices of $e \in \mc E$ lying in $\mc V_0$ and $\mc V_1$ respectively.

\begin{cor}[Abstract Mayer-Vietoris] \label{mv seq bipartite}
In the situation of Theorem~\ref{long ex sequence}, assume that the graph $\Gamma$ is bipartite, with respect to a partition 
$\mc V = \mc V_0 \sqcup \mc V_1$ of the set of vertices.  Then 
the long exact cohomology sequence in Theorem~\ref{long ex sequence} becomes the Mayer-Vietoris sequence
\[
\xymatrix{
0 \ar[r] & A(F) \ar[r]^-{\Delta} & \prod_{v \in \mc V_0} A(F_v) \times \prod_{v \in \mc V_1} A(F_v)
\ar[r]^-{-} & \prod_{e
\in \mc E} A(F_e) 
\ar `d[l] `[lld] [lld] \\
& H^1(F, A) \ar[r]^-{\Delta} & \prod_{v \in \mc V_0} H^1(F_v, A) \times \prod_{v \in \mc V_1} H^1(F_v, A) 
\ar[r]^-{-} & \prod_{e
\in \mc E} H^1(F_e, A)   
\ar `d[l] `[lld] [lld] \\ & H^2(F, A) 
\ar[r]^-{\Delta} & 
{\ \ \cdots \hskip 2.2in}
}
\]
where the maps $\Delta$ and $-$  are induced by the diagonal inclusion and by subtraction, respectively.
\end{cor}

\begin{thm}[Abstract Local-Global Principle] \label{local-global bipartite}
Fix a $\Gamma/F$-field $F_\bullet$, and fix a positive integer $m$ not
divisible by the characteristic of $F$. Suppose that the following
conditions hold:
\renewcommand{\theenumi}{\roman{enumi}}
\renewcommand{\labelenumi}{(\roman{enumi})}
\begin{enumerate}
\item \label{cond bip}
$\Gamma$ is bipartite, with respect to a partition 
$\mc V = \mc V_0 \sqcup \mc V_1$ of the set of vertices;
\item \label{cond comp}
for every $\xi \in \mc V \sqcup \mc E$, 
the cohomology of $\Z/m
\Z$ over $F_\xi$ is globally dominated;
\item \label{cond local factor}
given $v \in \mc V_0$, and elements $a_e \in F_e^\times$ for all $e \in \mc E$ that are incident to $v$, there exists $a \in F_v^\times$ such that $a_e/a \in (F_e^\times)^m$ for all $e$ (where we identify $F_v$ with its image $i_v^e(F_v) \subseteq F_e$).
\end{enumerate}
Then for all $n > 0$, the natural local-global maps
\[\sigma_n:H^{n+1}(F, \zmod m n) \to \prod_{v \in \mc V} H^{n+1}(F_v, \mbb
Z/m \Z (n))\]
are injective.
\end{thm}

\begin{proof}
Given hypothesis (\ref{cond bip}), as above we choose the orientation on $\Gamma$ such that $l(e) \in \mc V_0$ 
and $r(e) \in \mc V_1$ for all $e \in \mc E$.
Consider the homomorphisms:
\[\rho^{i,j}, \rho^{i,j}_0: \prod_{v \in \mc V} H^i(F_v, \zmod m
j) \to \prod_{e \in \mc E} H^i(F_e, \zmod m {j}),\]
where for $\alpha \in \prod_{v \in \mc V} H^i(F_v, \zmod m j)$
the $e$-th entries of $\rho^{i,j}(\alpha)$, $\rho_0^{i,j}(\alpha)$ are given by
\[\rho^{i,j}(\alpha)_e = (\alpha_{l(e))})_{F_e} - (\alpha_{r(e)})_{F_e},
 \ \ 
\rho^{i,j}_0(\alpha)_e = (\alpha_{l(e)})_{F_e}.
\]
Using hypothesis (\ref{cond comp}), Theorem~\ref{long ex sequence} allows us
to identify the kernel of
$\sigma_n$
with the cokernel of
\[\rho^{n, n} : \prod_{v \in \mc V} H^{n}(F_v, \zmod m n)
\to \prod_{e \in \mc E} H^{n}(F_e, \zmod m n).\] 
Thus it suffices to show that $\rho^{n, n}$ is surjective for $n \ge 1$.  
This in turn will follow from showing that $\rho^{n, n}_0$ is surjective, since the image of $\rho^{n, n}_0$ is contained in that of $\rho^{n, n}$
(using that $\Gamma$ is bipartite, and setting $\alpha_v=0$ for all $v \in \mc V_1$).

Writing $H^n_{\mc V} = \prod_{v \in \mc V} H^n(F_v, \zmod m {n})$ and $H^n_{\mc E} = \prod_{e \in \mc E}
H^n(F_e, \zmod m {n})$, we note that $\rho^{\bullet, \bullet}_0 : H^\bullet_{\mc V} \to H^\bullet_{\mc E}$ is a homomorphism of graded rings.
By hypothesis (\ref{cond local factor}), $\rho^{1, 1}_0$ is surjective,
since $H^1(E,\zmod m
1) = E^\times/(E^\times)^m$ for any field $E$ 
of characteristic not dividing $m$.
By Proposition~\ref{cup products},
every
element in $H^n_{\mc E}$ is a sum of $n$-fold products of elements in
$H^1_{\mc E}$, for $n \ge 1$. But since the map $\rho^{\bullet, \bullet}_0$ is a ring homomorphism, and $\rho^{1,1}_0$ is surjective, 
it follows that
$\rho^{n,n}_0$ is surjective as well for all $n \ge 1$.
\end{proof}

%%%%%%%%%%%%%%%%%%%
%%%%%%%%%%%%%%%%%%%
\section{Curves over complete discrete valuation rings} \label{curve section}
%%%%%%%%%%%%%%%%%%%
%%%%%%%%%%%%%%%%%%%

We now apply the previous general results to the more specific situation that we study in this paper: function fields $F$ over a complete discretely valued field $K$.  
In Section~\ref{curve patches} we will obtain a Mayer-Vietoris sequence and a local-global principle in the context of finitely many overfields $F_\xi$ of $F$ (``patches'').  
This can be compared with Theorem~3.5 of~\cite{HHK:H1}.
We will afterwards use that to obtain local-global principles with respect to the points on the closed fiber of a model (in Section~\ref{curve points}), and with respect to the discrete valuations on $F$ or on a regular model of $F$ (in Section~\ref{curve dvr}).   
These will later be used in Section~\ref{applications} to obtain applications to other algebraic structures.

%%%%%%%%%
\subsubsection{Notation}
%%%%%%%%%

We begin by fixing the standing notation for this section, which follows that of~\cite{HH}, \cite{HHK}, and \cite{HHK:H1}. 
Let $T$ be a complete discrete valuation ring with fraction field $K$ and
residue field $k$ and uniformizer $t$, and let $\mX$ be a projective,
integral and normal $T$-curve. Let $F$ be the function field of $\mX$. We
let $\cX$ be the closed fiber of $\mX$, and we choose a non-empty collection
of closed points $\P \subset \cX$, containing all the 
points at which distinct irreducible components of $\cX$ meet.
Thus the open complement $\cX
\smallsetminus \P$ is a disjoint union of finitely many irreducible affine $k$-curves $U$. Let $\UU$ denote the collection of
these open sets $U$.  

For a point $P \in \P$, we let $R_P$ be the local ring $\mc O_{\mX,
P}$ at $P$, and we let $\wh R_P$ be the completion at its maximal
ideal. Let $F_P$ be the fraction field of $\wh R_P$.  For a component
$U \in \UU$, we let $R_U$ be the subring of $F$ consisting of rational functions on $\mX$ that are regular at the points of $U$, i.e.\
\[R_U = \{f \in F | f \in \mc O_{\mX, Q} \text{ for all } Q \in U\}.\]
We also let $\wh R_U$ be the $t$-adic completion of $R_U$, and we let $F_U$ be the
field of fractions of $\wh R_U$.  Here $\wh R_P$ and $\wh R_U$ are Noetherian integrally closed domains (because $\wh X$ is normal), and in particular Krull domains.

For a point $P \in \P$ and a component $U \in \UU$, we say that $P$ and $U$
are \textit{incident} if $P$ is contained in the closure of $U$. Given $P \in \P$
and $U \in \UU$ that are incident, the 
prime ideal sheaf $\mathcal I$
defining the reduced closure $\ov U^{\rm red}$ of $U$ in $\mX$ induces a 
(not necessarily
prime) ideal $\mathcal I_P$ in the complete local ring $\wh R_P$. We call
the height one prime ideals of $\wh R_P$ containing $\mathcal I_P$ the
\textit{branches} on $U$ at $P$. We let $\B$ denote the collection of
branches on all points in $\P$ and all components in $\UU$. For a branch
$\wp$ on $P \in \P$ and $U \in \UU$, 
the local ring of $\wh R_P$ at $\wp$ is a discrete valuation ring $R_\wp$.  
Let $\wh R_\wp$ be its $\wp$-adic (or equivalently $t$-adic) completion, and let $F_\wp$ be the field of fractions of $\wh R_\wp$. Note that this is a
complete discretely valued field containing $F_U$ and $F_P$
(see \cite{HH}, Section~6, and \cite{HHK}, page 241).  

Associated to the curve $\mX$ and our choice of points $\P$, we
define a \textit{reduction graph} $\Gamma = \Gamma_{\mX, \P}$ whose vertex set is the disjoint
union of the sets $\P$ and $\UU$ and whose edge set is the set $\B$ of
branches. The incidence relation on this (multi-)graph, which makes it
bipartite, is defined by saying that an edge corresponding to a branch $\wp
\in \B$ is incident to the vertices $P \in \P$ and $U \in \UU$
if $\wp$ is a branch on $U$ at
$P$.  We choose the orientation on $\Gamma$ that is associated to the partition $\P \sqcup \UU$ of the vertex set.
We will consider the $\Gamma$-field $F_\bullet = F^{\mX, \P}_\bullet$ defined by
$F^{\mX, \P}_\xi = F_\xi$ for $\xi \in \P, \UU, \B$.

%%%%%%%%%%%%%%%%%%%
\subsection{Mayer-Vietoris and local-global principles with respect to patches} \label{curve patches}
%%%%%%%%%%%%%%%%%%%

Using the results of Section~\ref{abstract MV}, we now obtain the desired Mayer-Vietoris sequence for the $\Gamma$-field $F_\bullet$ that is associated as above to the function field $F$ and a choice of points $\P$ on the closed fiber of a normal model $\mX$ (see Theorem~\ref{meyer-veitoris for our field}).  In certain cases we show that this sequence splits into short exact sequences, possibly starting with the $H^2$ term (Corollaries~\ref{MV splits cyclic} and~\ref{MV splits tori}).  Related to this, we obtain a local-global principle for $H^n(F, \zmod m {n-1})$, in this patching context.

\begin{thm} \label{patching for our field} 
With $F$ and $F_\bullet$ as above, $F_\bullet$ is a $\Gamma/F$-field,
and patching holds for finite dimensional vector spaces over
$F_\bullet$.  Thus every linear algebraic group over $F$ is separably factorizable over $F_\bullet$.
\end{thm}

\begin{proof}
According to \cite[Corollary~3.4]{HHK:H1}, the fields $F_\xi$ for $\xi \in \P \sqcup \UU \sqcup \B$ form a factorization inverse system with inverse limit $F$.  That is, $F_\bullet$ is a $\Gamma/F$-field.  That 
result also asserts that patching holds for finite dimensional vector spaces over $F_\bullet$.  The assertion about being separably factorizable then follows from 
Corollary~\ref{patching yields sep fact}.
\end{proof}

%%%%%%%%%
\subsubsection{Global domination and Mayer-Vietoris}
%%%%%%%%%

The following result relies on a form of the Weierstrass Preparation Theorem that was proven in~\cite{HHK:Weier}, 
and which extended related results in~\cite{HH} and~\cite{HHK}.  
Another result that is similarly related to Weierstrass Preparation appears at Lemma~\ref{local Weier} below.

\begin{thm}[Global domination for patches]\label{weierstrass}
If $\xi \in \UU \sqcup \P \sqcup \B$ and if $m$ is a positive integer not divisible by $\cha(k)$, then 
the cohomology of $\Z/m
\Z$ over $F_\xi$ is globally dominated.
\end{thm}

\begin{proof} 
By Proposition~\ref{Kummer globally dominated}, it suffices to show that 
$(F_\xi\gd)^\times = ((F_\xi\gd)^\times)^m$.
So let $a \in (F_\xi\gd)^\times$.  Thus $a \in F_\xi F' \subseteq F_\xi\s$ for some finite separable extension $F'/F$.  
Let $\mX' \to \mX$ be the normalization of $\mX$ in $F'$, so that $\mX'$
is a normal projective $T$-curve with function field $F'$. Using
\cite[Lemma~6.2]{HH}, we may identify $F_\xi \otimes_F F'$ with $\prod_{\xi'}
F'_{\xi'}$, where $\xi'$ ranges through the points, components or
branches, respectively, lying above $\xi$ on $\mX'$. We also see by this
description that for each $\xi'$, the field $F'_{\xi'}$ is the
compositum of its subfields $F_\xi$ and $F'$.  Applying \cite[Theorems~3.3 and~3.7]{HHK:Weier} to the curve $\mX'$ and the field $F'_{\xi'}$, 
and again using \cite[Lemma~6.2]{HH}, it follows that
there is an \'etale cover $\mX''$ of $\mX'$
such that $a=bc^m$ for some $b \in F'' \subseteq F\s$ and 
$c \in F''_{\xi''} = F_\xi F'' \subseteq  F_\xi\gd$; here $F''$ is the function field of $\mX''$ and $\xi''$ is any point, component or branch, respectively, on $\mX''$ that lies over $\xi'$ on $\mX'$.  Now $b \in (F\s)^\times$, and $\cha(F)$ does not divide $m$, so $b \in ((F\s)^\times)^m$.  Thus $a \in ((F_\xi\gd)^\times)^m$.
\end{proof}

\begin{thm}[Mayer-Vietoris for Curves] \label{meyer-veitoris for our field} 
Let $A$ be a commutative linear algebraic group over $F$.  Assume that either
\renewcommand{\theenumi}{\roman{enumi}}
\renewcommand{\labelenumi}{(\roman{enumi})}
\begin{enumerate}
\item \label{our mv finite}
$A$ is finite of order not divisible by the characteristic of $k$; or
\item \label{our mv char 0}
$\cha(k)=0$.
\end{enumerate}
Then we have a long exact Mayer-Vietoris
sequence:
\vspace{-.1cm}
\[
\xymatrix{
0 \ar[r] & A(F) \ar[r]^-\Delta & \prod_{P \in \P} A(F_P) \times \prod_{U \in
\UU} A(F_U) \ar[r]^-{-} & \prod_{\wp \in \B} A(F_\wp) 
\ar `d[l] `[lld] [lld] \\
& H^1(F, A) \ar[r]^-\Delta & \prod_{P \in \P} H^1(F_P, A) \times \prod_{U \in
\UU} H^1(F_U, A) \ar[r]^-{-} & \prod_{\wp \in \B} H^1(F_\wp, A) 
\ar `d[l] `[lld] [lld] \\ & H^2(F, A) 
\ar[r]^-\Delta &
{\ \ \cdots \hskip 2.2in}
}
\]
\end{thm}

\begin{proof}
Let $\Gamma$ be the bipartite graph $\Gamma_{\mX, \P}$ as above.
By Theorem~\ref{patching for our field}, $A$ is separably factorizable over $F_\bullet$.
Now for each prime $\ell$ unequal to the characteristic of $k$, 
and each $\xi \in \UU \sqcup \P \sqcup \B$,
the cohomology of $\Z/\ell
\Z$ over $F_\xi$ is globally dominated, by Theorem~\ref{weierstrass}.
The conclusion now follows from Corollaries~\ref{les conditions} and~\ref{mv seq bipartite}.
\end{proof}

%%%%%%%%%
\subsubsection{Local-global principles with respect to patches}
%%%%%%%%%

\begin{lem} \label{local Weier}
Let $m$ be a positive integer that is not divisible by $\cha(K)$.  Let $P$ be a closed point of $X$, let $\wp_1,\dots,\wp_s$ be the branches of $X$ at $P$, and let $a_i \in F_{\wp_i}^\times$.  Then there exists $a \in F_P^\times$ such that $a_i/a \in (F_{\wp_i}^\times)^m$ for every $i$. 
\end{lem} 

\begin{proof}
Since $F_{\wp_i}$ is the $\wp_i$-adic completion of $F_P$, the Approximation Theorem \cite[VI.7.3, Theorem~2]{Bo:CA} implies that the elements $a_i$ can all be $\wp_i$-adically approximated arbitrarily well by an element $a \in F_P^\times$.  The result now follows by applying the strong form of Hensel's Lemma (see \cite[III.4.5, Corollary~1 to Theorem~2]{Bo:CA}) to the polynomials $Y^m - a_i/a \in \wh R_{\wp_i}[Y]$.
\end{proof}

\begin{thm}[Local-Global Principle] \label{patch local-global}
Let $\wh X$ be a normal projective curve over a complete discrete valuation ring $T$ with residue field $k$, let $\mc P$ be a non-empty finite subset of the closed fiber $X$ that includes the points at which distinct irreducible components of $X$ meet, and let $\mc U$ be the set of components of $X \smallsetminus \mc P$.  
Suppose that $m$ is an integer not divisible by the
characteristic of $k$. Then for each integer $n > 1$, the natural map
\[H^n(F, \zmod m {n-1}) \to \prod_{P \in \P} H^n(F_P, \zmod m {n-1}) \times
\prod_{U \in \UU} H^n(F_U, \zmod m {n-1})\]
is injective.
\end{thm}

\begin{proof}
The graph $\Gamma_{\mX, \P}$ is bipartite, with the set of vertices $\mc V$ partitioned as $\mc V_0 \sqcup \mc V_1$ with $\mc V_0 = \P$ and $\mc V_1 = \UU$.  So hypothesis~(\ref{cond bip}) of Theorem~\ref{local-global bipartite}
holds.  Hypothesis~(\ref{cond comp}) of that theorem, 
concerning global domination, also holds, by 
Theorem~\ref{weierstrass}.  Finally, hypothesis~(\ref{cond local factor}), 
in this case concerning the lifting of elements 
of the fields $F_{\wp_i}^\times$
to an element of $F_P^\times$ modulo $m$-th powers,
holds by Lemma~\ref{local Weier}.  Thus Theorem~\ref{local-global bipartite} applies, and the conclusion follows. 
\end{proof}

In some cases we can allow arbitrary Tate twists, and as a result the Mayer-Vietoris sequence splits into shorter exact sequences:

\begin{cor} \label{MV splits cyclic}
Let $m$ be an integer not divisible by the
characteristic of $k$, and suppose that the degree $[F(\bmu_m):F]$ is prime to $m$ (e.g.\ if $m$ is prime or $F$ contains a primitive $m$-th root of unity).  Let $r$ be any integer.  Then the Mayer-Vietoris sequence in Theorem~\ref{meyer-veitoris for our field} for $A = \zmod m r$ splits into exact sequences 
\[
\xymatrix{
0 \ar[r] & A(F) \ar[r] & \prod_{P \in \P} A(F_P) \times \prod_{U \in
\UU} A(F_U) \ar[r] & \prod_{\wp \in \B} A(F_\wp) 
\ar `d[l] `[lld] [lld] \\
& H^1(F, A) \ar[r] & \prod_{P \in \P} H^1(F_P, A) \times \prod_{U \in
\UU} H^1(F_U, A) \ar[r] & \prod_{\wp \in \B} H^1(F_\wp, A) \ar[r] &
0
}
\]
and
\[
\xymatrix{
0 \ar[r] & H^n(F, A) \ar[r] & \prod_{P \in \P} H^n(F_P, A) \times \prod_{U \in
\UU} H^n(F_U, A) \ar[r] & \prod_{\wp \in \B} H^n(F_\wp, A) \ar[r] &
0
}
\]
for all $n>1$.
\end{cor}

\begin{proof}
If $F$ contains a primitive $m$-th root of unity, then 
$A=\Z/m\Z = \zmod m {n-1}$ over $F$ and its extension fields, for all
$n$.  Hence in the Mayer-Vietoris sequence in
Theorem~\ref{meyer-veitoris for our field}(\ref{our mv finite}), the
maps $\iota_F : H^n(F,A) \to \prod_{P \in \mc P} H^n(F_P,A) \times \prod_{U \in \mc U} H^n(F_U,A)$ are injective for all $n>1$, by Theorem~\ref{patch local-global}.  The result now follows in this case.

More generally, let $F' = F(\bmu_m)$ and similarly for $F_P$ and
$F_U$. As above, $\iota_{F'}$ is injective. Using the naturality of
$\iota_F$ with respect to $F$, we have $\ker(\iota_F) \subseteq
\ker(\iota_{F'} \circ \res_{F'/F})$. Further, by the injectivity of
$\iota_{F'}$, $\ker(\iota_{F'} \circ \res_{F'/F}) = \ker(\res_{F'/F})
\subseteq \ker(\cores_{F'/F} \circ \res_{F'/F})$.  But $\cores \circ
\res:H^n(F,A) \to H^n(F,A)$ is multiplication by $[F':F]$ (\cite{GS},
Proposition~3.3.7), which is injective since $|A|=m$ and $[F':F]$ is
prime to $m$.  Thus these kernels are all trivial, and again the
result follows. 
\end{proof}

In Corollary~\ref{MV splits cyclic}, the initial six terms need {\em not} split into two three-term short exact sequences; i.e.\ the map on $H^1(F,A)$ need not be injective.  In fact, for $A = \Z/m \Z$ with $m>1$, a necessary and sufficient condition for splitting is that the reduction graph $\Gamma$ is a tree (\cite{HHK:H1}, Corollaries~5.6 and~6.4).  
But in the next result, there is splitting at every level.

\begin{cor} \label{MV splits tori}
Suppose that $\cha(k)=0$ and that $K$ contains a primitive $m$-th
root of unity for all $m \ge 1$. Then the Mayer-Vietoris sequence in
Theorem~\ref{meyer-veitoris for our field}(\ref{our mv char 0}) for
$\Gm$ splits into exact sequences 
\[
\xymatrix@C=.5cm{
0 \ar[r] & H^n(F, \Gm) \ar[r] & \prod_{P \in \P} H^n(F_P, \mbb
G_m) \times \prod_{U \in \UU} H^n(F_U, \Gm) \ar[r] & \prod_{\wp
\in \B} H^n(F_\wp, \Gm) \ar[r] &
0
}
\]
for all $n\ge 0$.
\end{cor}

\begin{proof}
By Theorem~\ref{meyer-veitoris for our field}(\ref{our mv char 0}),
it suffices to prove the injectivity of the maps
$H^n(F, \Gm) \to \prod_{P \in \P} H^n(F_P, \Gm) \times \prod_{U \in
\UU} H^n(F_U, \Gm)$ for all $n \ge 1$.
The case $n=1$ follows from the vanishing of $H^1(F,\Gm)$ by
Hilbert's Theorem~90.  It remains to show injectivity for $n>1$. 
Since $K$ contains all roots of unity, for each $m$ we may identify the Galois module $\Gm[m]=\bmu_m$ with $\Z/m
\Z$ and $\zmod m {n-1}$.

By Theorem~\ref{meyer-veitoris for our field}(\ref{our mv char 0}), 
the desired injectivity will follow from the surjectivity of the map
\[\prod_{P \in \P} H^{n-1}(F_P, \Gm) \times \prod_{U \in
\UU} H^{n-1}(F_U, \Gm) \to \prod_{\wp \in \B} H^{n-1}(F_\wp, \Gm).\] 
So let $\alpha \in \prod_{\wp \in \B} H^{n-1}(F_\wp, \Gm)$, and write $\alpha = (\alpha_\wp)_{\wp \in \mc B}$, with $\alpha_\wp \in H^{n-1}(F_\wp, \Gm)$.  
For each $\wp \in \mc B$, the element $\alpha_\wp$ is the image of
some $\til\alpha_\wp \in H^{n-1}(F_\wp, \bmu_{m_\wp})$ for some $m_\wp \ge
1$, by 
Lemma~\ref{torsion coho surjects}.
Since $\mc B$ is finite, we may let $m$ be the least common multiple of the integers $m_\wp$.  Thus $\alpha$ is the image of $\til\alpha = 
(\til \alpha_\wp)
\in \prod_{\wp \in \B} H^{n-1}(F_\wp, \bmu_m)
$.  

By Theorem~\ref{patch local-global} and Theorem~\ref{meyer-veitoris for our field}(\ref{our mv finite}), the map 
\[\prod_{P \in \P} H^{n-1}(F_P, \zmod m {n-1}) \times \prod_{U \in
\UU} H^{n-1}(F_U, \zmod m {n-1}) \to \prod_{\wp \in \B} H^{n-1}(F_\wp,\zmod m {n-1})\] 
is surjective.  
So by the identification $\bmu_m = \zmod m {n-1}$, it follows that  
$\til\alpha$ is the image of some element $\til\beta \in
\prod_{P \in \P} H^{n-1}(F_P, \bmu_m) \times \prod_{U \in
\UU} H^{n-1}(F_U, \bmu_m)$.
Let $\beta$ be the image of $\til\beta$ in $\prod_{P \in \P} H^{n-1}(F_P, \Gm) \times \prod_{U \in
\UU} H^{n-1}(F_U, \Gm)$.
Since the diagram
\[
\xymatrix{
\til\beta \in \prod_{P \in \P} H^{n-1}(F_P, \bmu_m) \times \prod_{U \in
\UU} H^{n-1}(F_U, \bmu_m) \ \ \ \ \ \ar[r] \ar[d] & \ \ \ \prod_{\wp \in \B} H^{n-1}(F_\wp,\bmu_m) \ni \til\alpha \ar[d] \\
\beta \in \prod_{P \in \P} H^{n-1}(F_P, \Gm) \times \prod_{U \in
\UU} H^{n-1}(F_U, \Gm) \ \ \ \ \ \ar[r]  & \ \ \ \prod_{\wp \in \B} H^{n-1}(F_\wp,\Gm) \ni \alpha
}
\]
commutes, $\beta$ maps to $\alpha$, as desired.
\end{proof}

Note that Corollaries~\ref{MV splits cyclic} and~\ref{MV splits tori} also provide patching results for cohomology, in addition to local-global principles.  Namely, for $n \ne 1$ in Corollary~\ref{MV splits cyclic}, or any $n$ in Corollary~\ref{MV splits tori}, those assertions show the following.  Given a collection of elements $\alpha_\xi \in H^n(F_\xi,A)$ for all $\xi \in \P \sqcup \UU$ such that $\alpha_P, \alpha_U$ induce the same element of $H^n(F_\wp,A)$ whenever $\wp$ is a branch on $U$ at $P$, there exists a unique $\alpha \in H^n(F,A)$ that induces all the $\alpha_\xi$.  In the situation of Theorem~\ref{meyer-veitoris for our field}, where splitting is not asserted, a weaker patching statement still follows: given elements $\alpha_\xi$ as above, there exists such an $\alpha$, but it is not necessarily unique.

%%%%%%%%%%%%%%%%%%%
\subsection{Local-global principles with respect to points} \label{curve points}
%%%%%%%%%%%%%%%%%%%

In this section we will investigate how to translate our results into
local-global principles 
in terms of the points on the closed fiber $X$ of $\mX$, rather than in terms of patches.  
Extending our earlier notation, if 
$P \in X$ is any point (not necessarily closed), we let $F_P$ denote the fraction field
of the complete local ring $\wh R_P := \wh {\mc O}_{\mX, P}$.  In 
particular, if $\eta$ is the generic point of an irreducible component $X_0$ of the closed fiber $X$, then $F_\eta$ is a complete discretely valued field, and it is 
the same as the $\eta$-adic
completion of~$F$. 

%%%%%%%%%
\subsubsection{The field $F_\eta^h$}
%%%%%%%%%

In order to bridge the gap between the fields $F_U$ and 
$F_\eta$, where $\eta$ is the generic point of 
the irreducible component $X_0 \subseteq X$ containing
$U$, we will consider a
subfield $F_\eta^h$ of $F_\eta$ that has many of the same properties
but is much smaller.  

Namely, with notation as above, let $R^h_\eta$
be the direct limit of the rings $\wh R_V$, where $V$ ranges over the
non-empty open subsets of $X_0$ that do not meet any other
irreducible component of $X$.  Equivalently, we may fix one such
non-empty open subset $U$, and consider the direct limit over the
non-empty open subsets $V$ of $U$.  Here $R^h_\eta$ is a subring of
$\wh R_\eta$; and we let $F_\eta^h$ be its fraction field.  Thus
$F_\eta^h$ is a subfield of $F_\eta$.

\begin{lem} \label{hens ring}
Let $X_0 \subseteq X$ be an irreducible component with generic point $\eta$, and let $U \subset X_0$ be a non-empty open subset meeting no other component.  
Then $R^h_\eta$ is a Henselian discrete valuation ring with 
respect to the $\eta$-adic valuation, having residue field $k(U)=k(X_0)$.  Its fraction field $F_\eta^h$ is the filtered direct limit of the fields $F_V$, where 
$V$ ranges over the non-empty open subsets of $U$.
\end{lem}

\begin{proof}
Each $F_V$ is contained in $F_\eta^h$, and every element of $F_\eta^h$ is of the form $a/b$ with $a,b$ in some common $F_V$.  So $F_\eta^h$ is the direct limit of the fields $F_V$.

Viewing $\eta$ as a prime ideal of $\wh R_V$, 
the fields $F_V$ each have a discrete valuation with respect to $\eta$, and these are compatible.  It follows that
$F_\eta^h$ is a discretely valued field with respect to the $\eta$-adic valuation.  We wish to show that the valuation ring of $F_\eta^h$ is $R^h_\eta$, with residue field $k(U)$. 
Note that the $t$-adic and $\eta$-adic metrics on $\wh R_V$ are equivalent,
since $\sqrt{(t)}=\eta$.

Since $\wh R_V$ is contained in the $\eta$-adic valuation ring of $F_V$, it follows that $R^h_\eta$ is contained in the valuation ring of $F_\eta^h$.  
To verify the reverse containment, consider a non-zero element $\alpha \in F_\eta^h$ with non-negative $\eta$-adic valuation.  
Thus $\alpha \in F_V^\times$ for some $V$; and so 
$\alpha = a/b$ with $a,b \in \wh R_V$ non-zero and $v_\eta(a) \ge v_\eta(b)$.
Since $\wh R_V$ is a Krull domain, the element
$b \in \wh R_V$ has a well defined divisor, which is a finite linear combination of prime divisors; and other than 
the irreducible closed fiber $V$ of $\Spec(\wh R_V)$, 
each of them has a locus that meets this closed fiber at only finitely many points.  
After shrinking $V$ by deleting these points, we may assume that $b$ is invertible in $\wh R_V[t^{-1}]$.  But also $v_\eta(a/b) \ge 0$;
and thus $a/b$ has no poles on $\Spec(\wh R_V)$.
So the element $\alpha = a/b \in F_V$ actually lies in $\wh R_V$, and hence in $R_\eta^h$ as desired.  
Thus 
$R^h_\eta$ is indeed the valuation ring of  $F_\eta^h$.  
Since the valuations on the rings $R_V$ are compatible and induce that of $R^h_\eta$, the maximal ideal $\eta R^h_\eta$ of $R^h_\eta$ is the direct limit of the prime ideals $\eta \wh R_V$ of the rings $\wh R_V$.
But $\wh R_V/\eta \wh R_V = k(V) = k(U)$ for all $V$.  
So the residue field of $R^h_\eta$ is $k(U)$.

It remains to show that $R_\eta^h$ is Henselian.  Let $S$ be a commutative \'etale algebra over~$R_\eta^h$, together with a section $\sigma: \eta \to \Spec(S)$ of $\pi:\Spec(S) \to \Spec(R_\eta^h)$ over the point $\eta$.  
To show that $R_\eta^h$ is Henselian,
we will check that $\sigma$ may be extended to a section over all of
$\Spec(R_\eta^h)$.
Now since $S$ is a finitely generated $R_\eta^h$-algebra, 
it is induced by an \'etale $\wh R_V$-algebra $S_V$ for some $V$,  together with a morphism $\pi_V:\Spec(S_V) \to \Spec(\wh R_V)$ that induces $\pi$ and 
has a section $\sigma_V^0:\eta \to \Spec(S_V)$ over the generic point $\eta$ of the closed fiber of $\Spec(\wh R_V)$.  
Here $\sigma_V^0$ defines a rational section over $V$, and hence a section over a non-empty affine open subset of $V$.  
So after shrinking $V$, we may assume that $\sigma_V^0$ is induced by a section $\sigma_V:V \to \Spec(S_V)$.  
But the ring $\wh R_V$ is $t$-adically complete; so by a version of Hensel's Lemma (Lemma~4.5 of \cite{HHK}) the section $\sigma_V$ over $V$ extends to a section of $\pi_V$, over all of $\Spec(\wh R_V)$.  
This in turn induces a section of $\pi$ over $\Spec(R_\eta^h)$ that extends $\sigma$, thereby showing that $R_\eta^h$ is Henselian.
\end{proof}

\begin{prop} \label{big patch local}
Let $\eta$ be the generic point of an irreducible component $X_0$ of $X$, and let $U$ be a non-empty affine open subset of $X_0$ that does not meet any other irreducible component of $X$.  Let $A$ be a smooth commutative group 
scheme over $F$.
Suppose $\alpha \in H^n(F_U, A)$ satisfies $\alpha_{F_\eta} = 0$. Then there is a Zariski open neighborhood $V$ of $\eta$ in $U$ such that $\alpha_{F_V} = 0$.
\end{prop}

\begin{proof}
The ring $\wh R_U$ is excellent, by \cite[Corollary 5]{Val} and regularity; hence so is its localization $(\wh R_U)_\eta$ at $\eta$.  The henselization $R_U^h$ 
of $(\wh R_U)_\eta$ contains $\wh R_U$, and its completion is $\wh R_\eta$; and it is minimal for these properties among henselian 
discrete valuation 
rings.  So $R_U^h$ 
is contained in $R_\eta^h$, and its fraction field $F_U^h$ is contained in $F_\eta^h$.
Let $c \in Z^n(F_U,A)$ represent the class $\alpha$.  
Since $\alpha_{F_\eta} = 0$, there is a finite Galois extension $L/F_\eta$ such that $c_{F_\eta}$ is the coboundary of a cochain in $C^{n-1}(L/F_\eta,A(L))$.  
This can be expressed by finitely many polynomial equations.  
By excellence, Artin Approximation (\cite[Theorem~1.10]{Artin Approx}) applies to $R_U^h$; and it follows that $c_{F_U^h}$ is the coboundary of an element of $C^{n-1}(F_U^h,A)$.  Thus $\alpha_{F_U^h}=0$ and hence $\alpha_{F_\eta^h}=0$.
The conclusion now follows from
Lemma~\ref{coho limit}, since $F^h_\eta$ is the filtered direct limit of the fields $F_V$, by the second part of Lemma~\ref{hens ring}.
\end{proof}

%%%%%%%%%
\subsubsection{Local-global principles with respect to points}
%%%%%%%%%

We now obtain a local-global principle in terms of points on the closed fiber $X$.

\begin{thm} \label{shapoints}
Let $A$ be a commutative linear algebraic group over $F$ and let $n > 1$.  Assume that either
\renewcommand{\theenumi}{\roman{enumi}}
\renewcommand{\labelenumi}{(\roman{enumi})}
\begin{enumerate}
\item \label{sha pts cyclic}
$A = \Z/ m \Z (r)$, where $m$ is an integer not divisible by $\cha(k)$, and where either $r=n-1$ or else $[F(\bmu_m):F]$ is prime to $m$; or
\item \label{sha pts char 0}
$A = \Gm$, $\cha(k)=0$, and $K$ contains a primitive $m$-th root of
unity for all $m \ge 1$.
\end{enumerate}
Then  
the
natural map
\[H^n(F, A) \to \prod_{P \in \cX} H^n(F_P, A)\]
is injective, where $P$ ranges through all the points of
the closed fiber.
\end{thm}

\begin{proof}
Let $\alpha \in H^n(F,A)$ be an element of the above kernel.  
Consider the irreducible components $X_i$ of $X$, and their generic points $\eta_i \in X_i \subseteq X$.  Thus $\alpha_{F_{\eta_i}} = 0$ for each $i$ (taking $P = \eta_i$).
By Proposition~\ref{big patch local}, we may choose a non-empty Zariski affine open subset $U_i \subset X_i$, not meeting any other component of $X$, such that 
$\alpha_{F_{U_i}}$ is trivial.  Let $\UU$ be the collection of these sets $U_i$, and let $\P$ be the complement in $X$ of the union of the sets $U_i$.  Then $\alpha$ is in the kernel of the map on $H^n(F,A)$ in Theorem~\ref{patch local-global}, Corollary~\ref{MV splits cyclic}, or Corollary~\ref{MV splits tori} respectively.  Since that map is injective, it follows that $\alpha=0$.
\end{proof}

%%%%%%%%%%%%%%%%%%%
\subsection{Local-global principles with respect to discrete
valuations} \label{curve dvr}
%%%%%%%%%%%%%%%%%%%

Using the previous results, we now investigate how to translate our
results into local-global principles involving discrete valuations on
our field $F$, and in particular those valuations arising from
codimension one points on our model $\mX$ of $F$.  Our main result here is
Theorem~\ref{sha_vanish}, which parallels
Theorem~\ref{shapoints}(\ref{sha pts cyclic}), and asserts the
vanishing of the obstruction $\Sha^n(F,A)$ to such a local-global
principle, for $n>1$ and $A$ an appropriate twist of $\Z/m\Z$.  

In the case $n=1$, a related result appeared
at~\cite[Corollary~8.11]{HHK:H1}, but with different hypotheses and
for different groups.  In fact, for a constant finite group $A$, the
obstruction $\Sha^1(F,A)$ is non-trivial unless the reduction graph
$\Gamma$ of a regular model $\mX$ of $F$ is a tree; see \cite{HHK:H1},
Proposition~8.4 and Corollary~6.5.  (As in \cite{HHK:H1}, ``discrete valuations'' are required to have value group isomorphic to $\mbb Z$, and in particular to be non-trivial.)

For the remainder of this section we make the {\em standing assumption
that $\wh X$ is regular}.  

\begin{lem} \label{branch extension}
Let $P$ be a point of $X$ and let $v$ be a discrete valuation on $F_P$.  Then the restriction $v_0$ of $v$ to $F$ is a discrete valuation on $F$.  Moreover if $v$ is induced by a codimension one point of $\Spec(\wh R_P)$ (or equivalently, a height one prime of $\wh R_P$), then $v_0$ is induced by a codimension one point of $\wh X$ whose closure contains $P$.
\end{lem}

\begin{proof}
The first assertion is given at \cite[Proposition~7.5]{HHK:H1}.  For the second assertion, if $v$ is induced by a height one prime of $\wh R_P$, then $\wh R_P$ is contained in the valuation ring of $v$.  Hence so is the local ring $R_P$, which is then also contained in the valuation ring of $v_0$.  Thus $v_0$ is induced by a codimension one point of $\Spec(R_P)$, and so by a codimension one point of $\wh X$ whose closure contains $P$.
\end{proof}

Given a field $E$, let $\Omega_E$ denote the set of discrete valuations on $E$.  For $v \in \Omega_E$, write $E_v$ for the $v$-adic completion of $E$.
If $A$ is a commutative group scheme over $E$, let
\[\Sha^n(E,A) = \ker\biggl(H^n(E,A) \to
\prod_{v \in \Omega_E} H^n(E_v, A)\biggr).\]

Similarly, given a normal integral scheme $Z$ with function field $E$, let $\Omega_Z \subseteq \Omega_E$ denote the subset consisting of the discrete valuations on $E$ that correspond to codimension one points on $Z$.  If $A$ is as above, let   
\[\Sha^n_Z(E,A) = \ker\biggl(H^n(E,A) \to
\prod_{v \in \Omega_Z} H^n(E_v, A)\biggr).\]
Here if $Z = \Spec(R)$, we also write $\Sha^n_R(E,A)$ for $\Sha^n_Z(E,A)$.

We will be especially interested in the case that $E=F$, the function
field of a regular projective curve $\mX$ over our complete discrete valuation ring $T$; and where $Z$ is either $\mX$ or $\Spec(\wh R_P)$ for some closed point $P \in \mX$.  

In the case that $Z=\mX$, with closed fiber $X$ and function field $F$ as before, there is a related group
\[\Sha_{\mX,X}^n(F,A) = \ker\biggl(H^n(F,A) \to
\prod_{P \in X} H^n(F_P, A)\biggr).\]
Note that $\Sha_{\mX,X}^n(F,A)$ is contained in $\Sha_\mX^n(F,A)$ by 
\cite[Proposition~7.4]{HHK:H1}, which asserts that every field of the form $F_v$ contains a field of the form $F_P$.  In the above notation, Theorem~\ref{shapoints} asserts that $\Sha_{\mX,X}^n(F,A) = 0$ if condition (\ref{sha pts cyclic}) or (\ref{sha pts char 0}) of that result is satisfied.

%%%%%%%%%
\subsubsection{Relating local-global obstruction on a regular model to
obstructions at closed points}
%%%%%%%%%

A key step in relating our patches to discrete valuations is the
following result, which parallels Proposition~8.4 of \cite{HHK:H1}.  That result considered only the case $n=1$, but did not require the linear algebraic group to be commutative (since $H^1$ is defined even for non-commutative groups).

Here $X_{(0)}$ denotes the set of closed points of $X$, and $\prod'$ denotes the restricted product, i.e.\ the subgroup of the product consisting of elements in which all but finitely many entries are trivial.

\begin{prop} \label{sha ses}
Let $A$ be a linear algebraic group over $F$.
\renewcommand{\theenumi}{\alph{enumi}}
\renewcommand{\labelenumi}{(\alph{enumi})}
\begin{enumerate}
\item \label{left exact sha seq}
The natural map 
\(H^n(F, A) \to \prod_{P \in \cX} H^n(F_P, A)\) induces an exact sequence 
\[0 \to \Sha_{\mX,X}^n(F,A) \xrightarrow{\iota} \Sha_\mX^n(F,A)  \xrightarrow{\phi}  \prod_{P \in X_{(0)}}{\!\!\!\!'} \  \Sha^n_{\wh R_P}(F_P,A).\]
\item \label{exact sha seq}
If $A$ is finite and of order not divisible by $\cha(k)$, or if $\cha(k)=0$, the exact sequence extends to
\[0 \to \Sha_{\mX,X}^n(F,A) \xrightarrow{\iota} \Sha_\mX^n(F,A)  \xrightarrow{\phi}  \prod_{P \in X_{(0)}}{\!\!\!\!'} \  \Sha^n_{\wh R_P}(F_P,A) \to 0. \]
\item \label{shasha}
If $n>1$ and $A$ satisfies hypothesis~(\ref{sha pts cyclic}) or~(\ref{sha pts char 0})
of Theorem~\ref{shapoints}, then $\phi$ is an isomorphism.
\end{enumerate}
\end{prop}

\begin{proof}
It follows from Lemma~\ref{branch extension} that for each $P \in X_{(0)}$, the image of
$\Sha_\mX^n(F,A)$ under $H^n(F, A) \to H^n(F_P,A)$ lies in $\Sha^n_{\wh R_P}(F_P,A)$.  Thus we obtain a group homomorphism 
$\Sha_\mX^n(F,A) \to \prod_{P \in X_{(0)}} \Sha^n_{\wh R_P}(F_P,A)$.  To obtain the map $\phi$, we 
wish to show that the image is contained in the restricted product.  
If $\alpha \in \Sha_\mX^n(F,A) \subseteq H^n(F,A)$
and $\eta$ is the generic point of an irreducible component of $X$,
then the image of 
$\alpha$ in $H^n(F_\eta,A)$ is trivial, since $\eta$ is a 
codimension one point on $\wh X$.  
By Proposition~\ref{big patch local}, 
$\alpha$ has trivial image in $H^n(F_U,A)$
for some Zariski open neighborhood $U$ of $\eta$, 
and hence in $H^n(F_P,A)$ for each $P \in U$.  
The union of these sets $U$, as $\eta$ varies, contains all but finitely many closed points of $X$.  So indeed the image of $\alpha$ lies in the restricted product.

The composition $\phi \iota$ is trivial by definition of $\Sha_{\mX,X}^n(F,A)$.
To complete the proof of part~(\ref{left exact sha seq}), 
let $\alpha \in \Sha_\mX^n(F,A)$
be any element in the kernel of this map.  Then the image of $\alpha$ in $\Sha^n_{\wh R_P}(F_P,A) \subseteq H^n(F_P,A)$
is trivial for every closed point $P$ on the closed fiber $X$.  
Meanwhile, for any non-closed point $\eta$ of $X$ (viz.\ the generic point of an irreducible component of $X$), 
the image of $\alpha$ in $H^n(F_\eta, A)$
is also trivial, by the definition of $\Sha^n$, 
since $\eta$ is a codimension one point of $\wh X$.
Hence $\alpha$ lies in $\Sha_{\mX,X}^n(F,A) \subseteq \Sha_\mX^n(F,A)$, as required.

To prove part~(\ref{exact sha seq}), i.e.\
that $\phi$ is surjective, take an element $(\alpha_P)_{P \in X_{(0)}}
$ in the above restricted product.  
Thus $\alpha_P =0$ for all $P \in X_{(0)}$ outside of some finite set $\mc P$ that can be chosen to include the points where distinct components of $X$ meet.  Since $\alpha_P \in  \Sha^n_{\wh R_P}(F_P,A)$, its image in $H^n(F_\wp,A)$ is trivial for every branch $\wp$ of $X$ at $P$.  
Let $\mc U$ be the set of components of the complement of $\mc P$ in $X$, and let $\alpha_U = 0$ for each $U \in \mc U$. 
The hypotheses of Theorem~\ref{meyer-veitoris for our field} are satisfied in our situation, and the exact sequence there yields that the tuple $(\alpha_\xi)_{\xi \in \mc P \sqcup \mc U}$ is of the form $\Delta(\alpha)$ for some $\alpha \in H^n(F,A)$.  The image of $\alpha$ under \(H^n(F, A) \to \prod_{P \in \cX} H^n(F_P, A)\) is 
$(\alpha_P)_{P \in X_{(0)}} \in \prod'_{P \in X_{(0)}}
\Sha^n_{\wh R_P}(F_P,A)$.
To complete the proof of (\ref{exact sha seq}), we show that 
$\alpha \in \Sha_\mX^n(F,A)$, i.e.\ $\alpha_v = 0$ for each $v \in \Omega_{\wh X}$.  By \cite[Proposition~7.4]{HHK:H1}, $F_v$ contains $F_P$ for some $P \in X$.  If $P$ is a closed point, then the discrete valuation on $F_v$ restricts to a discrete valuation $v_P$ on $F_P$ (which in turn restricts to $v$ on $F$).  But $\alpha_P \in \Sha^n_{\wh R_P}(F_P,A)$, so $\alpha$ becomes trivial over $(F_P)_{v_P}$ and hence over $F_v$.  If instead $P$ is a point of codimension one, i.e.\ the generic point $\eta$ of some $U \in \mc U$, then $v=v_\eta$, and $\alpha_v=0$ because $\alpha_U = 0$ and $F_U \subset F_\eta$.

Part~(\ref{shasha}) now follows from part~(\ref{exact sha seq}) and Theorem~\ref{shapoints}, which says $\Sha_{\mX,X}^n(F,A)=0$.
\end{proof}

%%%%%%%%%
\subsubsection{Local-global principles at closed points}
%%%%%%%%%

We will use the following statement of Panin which asserts a particular case
of the analog of the Gersten conjecture in the context of the theory
of Bloch and Ogus.  Here $\kappa(z)$ denotes the residue field at a point $z$, and $Z^{(i)}$ denotes the set of points of $Z$ having codimension $i$.
As usual, $\zmod m {-r}$ denotes $\Hom(\zmod m r,\Z/m\Z)$ for $r>0$, where 
$m$ is not divisible by the characteristic of the field.  Also, for $m$ as above and for any $r \in \mbb Z$, if $A$ is an $m$-torsion group scheme then $A(r)$ denotes $A \otimes \zmod m r$. 

\begin{thm}[{\cite [Theorem~C]
{Panin:EGC}}] \label{gersten}
Suppose that $R$ is an equicharacteristic regular local ring with
fraction field $F$, and let $Z = \Spec(R)$.
Then for any positive integer $m$ that is not divisible by the characteristic, and any $m$-torsion finite \'etale commutative group scheme $A$ over $R$, the Cousin complex
\[0 \to H^n(Z, A) \to H^n(F, A)
\to \bigoplus\limits_{z \in
Z^{(1)}}H^{n-1}(\kappa(z), 
A(-1)) \to \bigoplus\limits_{z \in
Z^{(2)}}H^{n-2}(\kappa(z), 
A(-2)) \to
\cdots\]
of \'etale cohomology groups
is exact.
\end{thm}

\begin{prop} \label{gersten local Sha vanish}
Under the hypotheses of Theorem~\ref{gersten}, assume that $R$ is complete.  Then
$\Sha^n(F,A)=\Sha_Z^n(F,A)=0$ for $n\ge 1$.
\end{prop}

\begin{proof}
Let $d$ be the Krull dimension of $R$.  The assertion is trivial if $d \le 1$, so we may assume $d \ge 2$.  Since $\Sha^n(F,A) \subseteq \Sha_Z^n(F,A)$, it suffices to show the vanishing of the latter group.

Let $\alpha \in \Sha_Z^n(F,A) \subseteq H^n(F,A)$.
Consider the exact sequence in Theorem~\ref{gersten}.  For each $z \in Z^{(1)}$, 
the ramification map
$H^n(F, A)
\to H^{n-1}(\kappa(z), 
A(-1))$
factors through the map to the completion $H^n(F_z,A)$.  But the image of $\alpha$ in $H^n(F_z,A)$ vanishes, since $\alpha \in \Sha_Z^n(F,A)$.  Hence $\alpha$ maps to zero in $\bigoplus_{z \in
Z^{(1)}}H^{n-1}(\kappa(z), 
A(-1))$, and thus it
is induced by a class $\til \alpha
\in H^n(Z, A)$.

Let $k'$ be the residue field of $R$ at its maximal ideal (corresponding to the closed point of $Z$).  Let $\sigma_1,\dots,\sigma_d$ be a regular system
of parameters in $R$.  Write 
$\sigma=\sigma_1$ and write $\wh R_\sigma$ for the 
completion of the local ring of $R$ at the prime ideal $(\sigma)$.
Thus $\wh R_\sigma$ is a complete discrete valuation ring with uniformizer $\sigma$; 
let $F_\sigma$ and $\kappa(\sigma)$ denote its fraction field and residue field, respectively.
Here $\kappa(\sigma)$ is the fraction field of
$\mc O_{\kappa(\sigma)} := R/(\sigma)$, an equicharacteristic 
regular complete local ring of dimension $d-1$, 
such that the residues $\bar \sigma_2,\dots,\bar \sigma_d$ of 
$\sigma_2,\dots,\sigma_d$ form a regular system of parameters.
By Theorem~\ref{gersten}, 
the natural maps 
\[H^n(\wh R_\sigma, A) \to H^n(F_\sigma, A) \ \ \ {\rm and} \ \ \ \ H^n(\mc O_{\kappa(\sigma)}, A) \to H^n(\kappa(\sigma), A)\]
are injections.  The complete local rings $R$ and
$\mc O_{\kappa(\sigma)}$ each have
residue field $k'$, so by 
\cite[Theorem~III.4.9]{Artin}
we may identify
\[H^n(Z, A) = H^n(k', A) = H^n(\mc O_{\kappa(\sigma)}, A),\]
via restriction to the closed point.
The natural map $H^n(Z,A) \to H^n(\mc O_{\kappa(\sigma)}, A)$ is thus an isomorphism.
We have the following commutative diagram:
\[\xymatrix{
H^n(\mc O_{\kappa(\sigma)}, A) \ar@{^(->}[r] & H^n(\kappa(\sigma), A) \\
{\!\!\! \!\!\! \tilde \alpha \in H^n(Z, A) \ } \ar[u]^{\wr} \ar[r] \ar@{^(->}[d] & 
{\ H^n(\wh R_\sigma, A) \ } \ar[u] \ar@{^(->}[d] \\
{\!\!\! \!\!\! \alpha \in H^n(F, A) \ } \ar[r] & 
{\ H^n(F_\sigma,A) \ }
}\]
Since $\sigma$ defines a codimension one point of $Z$, the image
of $\alpha \in \Sha_Z^n(F,A) \subseteq H^n(F, A)$ in $H^n(F_\sigma,A)$ is trivial.  Since $\tilde \alpha$ maps to $\alpha$, a diagram chase then shows that $\tilde \alpha$ is trivial and hence so is $\alpha$.
\end{proof}

In our situation, with $R = \wh R_P$ arising from a regular model $\wh X$,
Proposition~\ref{gersten local Sha vanish} asserts:

\begin{cor} \label{local sha vanish}
Suppose that $K$ is an equicharacteristic complete discretely valued field of
characteristic not dividing~$m$, and that $\mX$ is regular.  Then for every $P \in X$ and 
$m$-torsion finite \'etale commutative group scheme $A$ over $\wh R_P$, 
$\Sha^n(F_P,A)=\Sha^n_{\wh R_P}(F_P,A) = 0$.
\end{cor}

%%%%%%%%%
\subsubsection{Local-global principles for function fields}
%%%%%%%%%

Finally, we obtain our local-global principles over our field $F$ with respect to discrete valuations:

\begin{thm} \label{sha_vanish}
Suppose that $K$ is an equicharacteristic complete discretely valued field of
characteristic not dividing $m$, and that $\mX$ is regular. 
Let $n>1$.  Then 
\[\Sha^n(F,\zmod m {n-1}) = \Sha_\mX^n(F,\zmod m {n-1}) = 0.\]  
If 
$[F(\bmu_m):F]$ is prime to $m$
then also $\Sha^n(F,\zmod m r)=\Sha_\mX^n(F,\zmod m r)=0$ for all $r$.
\end{thm}

\begin{proof}
In each of the two cases considered, hypothesis~(\ref{sha pts cyclic}) of Theorem~\ref{shapoints} is satisfied.  Since $n>1$, 
Proposition~\ref{sha ses}(\ref{shasha}) then applies.  
The theorem now follows by Corollary~\ref{local sha vanish} and the containment $\Sha^n(F,A) \subseteq \Sha_\mX^n(F,A)$.
\end{proof}

\begin{rem} 
The case $n=2$, concerning Brauer groups, holds even without assuming equal characteristic
(\cite[Theorem~4.3(ii)]{CPS}, \cite[Corollary~9.13]{HHK:H1}).  It would be interesting to know if the same is true for $n>2$, and also if 
Theorem~\ref{sha_vanish} 
has an analog for $\Gm$ in characteristic zero as in 
Theorem~\ref{shapoints}(\ref{sha pts char 0}).  But carrying over the above proof would require versions of 
Panin's result \cite[Theorem~C]{Panin:EGC} in those situations.
\end{rem}

%%%%%%%%%%%%%%%%%%%
%%%%%%%%%%%%%%%%%%%
\section{Applications to torsors under noncommutative groups} \label{applications}
%%%%%%%%%%%%%%%%%%%
%%%%%%%%%%%%%%%%%%%

As an application of our results, in this section we give local-global principles for $G$-torsors over $F$
for certain connected noncommutative linear algebraic groups
$G$, and for related structures.  
Our method is to use cohomological invariants in order to reduce to our local-global principles in Galois cohomology (viz.\ to Theorems~\ref{shapoints}(\ref{sha pts cyclic}) and~\ref{sha_vanish}).

We preserve the notation and terminology established at the beginning of 
Section~\ref{curve section}.  In particular, we write $T$ for the valuation ring of $K$, and $k$ for the residue field.  We let $\wh X$ be a normal, integral projective curve over $T$, with closed fiber $X$ and function field $F$.  As before, we write $\Omega_F$ for the set of discrete valuations on the field $F$, and write $\Omega_\mX$ for the subset of $\Omega_F$ consisting of those discrete valuations that arise from codimension one points on $\mX$.

%%%%%%%%%%%%%%%%%%%
\subsection{Relation to prior results}
%%%%%%%%%%%%%%%%%%%

The basic strategy used in this section to obtain local-global principles
for torsors was previously used in~\cite[Theorem~5.4]{CPS}, to obtain
a local-global principle for $G$-torsors over the function field $F$
of a smooth projective geometrically integral curve over a $p$-adic
field $K$, where $G$ is a linear algebraic $F$-group that is
quasisplit, simply connected, and absolutely almost simple without an
$E_8$ factor. There they used the local-global principle of Kato for
$H^3$ together with the fact that the fields under their consideration
were of cohomological dimension three. Our applications arise from our
new local-global principles for higher cohomology groups, and hence
do not require any assumptions on cohomological dimension.

Local-global principles for $G$-torsors were also obtained
in~\cite{HHK:H1} (as well as in~\cite{HHK}, in the context of
patches).  But there the linear algebraic groups $G$ were required to
be rational varieties, whereas here there is no such hypothesis.  On
the other hand, here we will be looking at specific types of groups,
such as $E_8$ and $F_4$.  Another difference is that
in~\cite{HHK:H1}, in order to obtain local-global principles with
respect to discrete valuations, we needed to make additional
assumptions (e.g.\ that $k$ is algebraically closed of characteristic
zero, or that $G$ is defined and reductive over $\mX$;
see~\cite[Corollary~8.11]{HHK:H1}).  Here the only assumption needed
for local-global principles with respect to discrete valuations is
that $K$ is equicharacteristic.  (If we wish to consider only those
discrete valuations that arise from a given model $\mX$ of $F$, then
we also need to assume that $\mX$ is regular.)  Thus even in the cases
where the groups considered below are rational, the results here go
beyond what was shown for those groups in~\cite{HHK:H1}.

%%%%%%%%%%%%%%%%%%%
\subsection{Injectivity vs.\ triviality of the kernel}
%%%%%%%%%%%%%%%%%%%

The local-global principles for $G$-torsors will be phrased in terms
of local-global maps on $H^1(F,G)$.  Because of non-commutativity,
$H^1(F,G)$ is just a pointed set, not a group.  Thus there are two
distinct questions that can be posed about a local-global map: whether
the kernel is trivial, and whether the map is injective (the latter
condition being stronger).  And as in Section~\ref{curve section},
there are actually several local-global maps:  
$H^1(F,G) \to \prod_{v \in \Omega_F} H^1({F_v},G)$,  
$H^1(F,G) \to \prod_{v \in \Omega_\mX} H^1({F_v},G)$, and
$H^1(F,G) \to \prod_{P \in X} H^1({F_P},G)$,
with kernels $\Sha(F,G)$,
$\Sha_\mX(F,G)$, and $\Sha_{\mX,X}(F,G)$ respectively.  
(As is common, here we write $\Sha$ for $\Sha^1$.)
Note
that if $\Sha_\mX(F,G)=0$ for some model $\mX$ then $\Sha(F,G)=0$; and
similarly for injectivity of the corresponding maps.  So we will
emphasize the cases of $\Sha_\mX(F,G)$ and $\Sha_{\mX,X}(F,G)$.

%%%%%%%%%%%%%%%%%%%
\subsection{Local-global principles via cohomological invariants}
%%%%%%%%%%%%%%%%%%%

The approach that we take here for obtaining our applications is to
use cohomological invariants of algebraic objects.

Recall that an \textit{invariant} over $F$ is a morphism of functors
$a:S \to H$, where $S:({\rm Fields}/F) \to ({\rm Pointed\ Sets})$ and
$H:({\rm Fields}/F) \to ({\rm Abelian\ Groups}) ($\cite{GMS}, Part~I,
Sect.~I.1). Most often, as in~\cite{GMS}, $S$ will have the form $S_G$
given by $S_G(E) = H^1(E,G)$ for some linear algebraic group $G$ over $F$; this classifies $G$-torsors over $E$, and also often
classifies other types of algebraic structures over $F$.  In practice,
$H(E)$ will usually take values in Galois cohomology groups of the form
$H^n(E,\Z/m\Z(n-1))$.

The simplest situation is described in the following general result,
where we retain the standing hypotheses stated at the beginning of 
Section~\ref{curve section}, with $\mX$ a normal model of $F$.

\begin{prop} \label{invariant local-global}
Let $a:S \to H$ be a cohomological invariant over $F$, where $H(E) =
H^n(E,\Z/m\Z(r))$ for some integers $n,m,r$ with $n,m$ positive, and
where $m$ is not divisible by $\cha(k)$.  Assume either that $r=n-1$,
or else that the degree $[F(\bmu_m):F]$ is prime to $m$.

\renewcommand{\theenumi}{\alph{enumi}}
\renewcommand{\labelenumi}{(\alph{enumi})}
\begin{enumerate}
\item \label{loc-gl trivial kernel}
If $a(F):S(F) \to H(F)$ has trivial kernel, 
then so does
the local-global map \(S(F) \to \prod_{P \in \cX} S(F_P)\).
Moreover, if $K$ is equicharacteristic and $\mX$ is regular, then the same holds for 
 \(S(F) \to
\prod_{v \in \Omega_{\wh X}}  S(F_v)\).
\item \label{loc-gl injective}
If $a(F):S(F) \to H(F)$ is injective, then so is the local-global map \(S(F) \to \prod_{P \in \cX} S(F_P)\).  If in addition $K$ is equicharacteristic and $\mX$ is regular, then \(S(F) \to
\prod_{v \in \Omega_{\wh X}}  S(F_v)\) is injective as well.
\end{enumerate}
\end{prop}

\begin{proof}
Consider the commutative diagrams
\[\xymatrix@C=1.3cm @R=.5cm{
S(F) \ \ar[r]^-{a(F)} \ar[d] & \ H(F) \ar[d] & 
S(F) \ \ar[r]^-{a(F)} \ar[d] & \ H(F) \ar[d] \\
\prod_{P \in \cX} S(F_P) \ \ \ar[r]^-{\prod a(F_P)} & \ \ \prod_{P \in
\cX} H(F_P) & 
\prod_{v \in \Omega_{\mX}} S(F_v) \ \ \ar[r]^-{\prod a(F_v)} & \ \
\prod_{v \in \Omega_{\mX}} H(F_v). 
}\]
The result follows by a diagram chase, using the fact that the
right-hand vertical map in the first diagram is injective by
Theorem~\ref{shapoints}(\ref{sha pts cyclic}), and that the
corresponding map in the second diagram is injective in the case that
$K$ is equicharacteristic and $\mX$ is regular, by
Theorem~\ref{sha_vanish}.  
\end{proof}

Recall that a linear algebraic group $G$ over $F$ is \textit{quasi-split} if it has a Borel subgroup defined over $F$.  It is \textit{split} if it has a Borel subgroup over $F$ that has a composition series whose successive quotient groups are each isomorphic to $\Gm$ or $\Ga$.  If $G$ is reductive, this is equivalent to $G$ having a maximal torus that is split (i.e.\ a product $\Gm^n$).

\begin{cor} \label{rost local-global}
Let $G$ be a simply connected linear algebraic group over $F$.
Consider the Rost invariant $R_G:H^1(*,G) \to H^3(*,\Z/m\Z(2))$
of $G$, and assume that the characteristic of $k$ does not divide $m$.
In each of the following cases, $\Sha_{\mX,X}(F,G) = 0$.  
If $K$ is equicharacteristic then $\Sha(F,G) = 0$; and $\Sha_\mX(F,G) = 0$ if in addition the given model $\wh X$ is regular.
\begin{itemize}
\item $G$ is a quasi-split group of type $E_6$ or $E_7$.
\item $G$ is
an almost simple group that is quasi-split of absolute rank at most $5$.
\item $G$ is
an almost simple group that is quasi-split of type $B_6$ or $D_6$.
\item $G$ is an almost simple group that is split of type $D_7$.
\end{itemize}
\end{cor}

\begin{proof}
In each of these cases, the Rost invariant $R_G$ has trivial kernel.  This is by \cite[Main
Theorem~0.1]{Gar} in the first case, and by \cite[Theorem~0.5]{Gar} in
the other cases.  So the assertion follows from
Proposition~\ref{invariant local-global}(\ref{loc-gl trivial kernel}).
\end{proof}

\begin{cor} \label{SL loc-glob}
Let $m$ be a square-free positive integer that is not divisible by the characteristic of $k$, and let $A$ be a central simple $F$-algebra of degree $m$.
Then 
the local-global map \(H^1(F, \SL_1(A)) \to \prod_{P \in \cX} H^1(F_P, \SL_1(A))\) is injective.  If in addition $K$ is equicharacteristic and $\mX$ is regular, then 
the map \(H^1(F,\SL_1(A)) \to
\prod_{v \in \Omega_{\wh X}} H^1(F_v, \SL_1(A))\) is injective.
\end{cor}

\begin{proof}
By~\cite[12.2]{MerSus} (see also~\cite[7.2]{Se:Bo}), given a division algebra $A$ of degree $m$, there is a cohomological invariant 
$a:H^1(*,\SL_1(A)) \to H^3(*,\Z/m\Z(2))$ that is injective if $m$ is square-free.  So the result follows from Proposition~\ref{invariant local-global}(\ref{loc-gl injective}).
\end{proof}

In particular, $\Sha_{\mX,X}(F, \SL_1(A))$ and $\Sha(F, \SL_1(A))$ respectively vanish in the above situations.  Also, via the identification of 
$H^1(F, \SL_1(A))$ with $F^\times/\Nrd(A^\times)$, the above result gives a local-global principle for elements of $F^\times$ to be reduced norms from a (central) division algebra $A$; cf.\ also \cite[p.~146]{Kato}.

Other applications can be obtained by using a combination of cohomological invariants.  This is done in the next results.

\begin{prop} \label{loc glob E8}
Let $G$ be a
simple linear algebraic group of type $E_8$ over $F$.
\renewcommand{\theenumi}{\alph{enumi}}
\renewcommand{\labelenumi}{(\alph{enumi})}
\begin{enumerate}
\item \label{odd degree E8}
Assume $\cha(K) = 0$.  Then the group $G$ is split over some odd degree extension of $F$ if and only if $G_{F_P}$ is split over some odd degree extension of $F_P$ for every $P \in X$. 
\item \label{prime to 5 E8}
Assume $\cha(K) \ne 2,3,5$.  Then the same holds for extensions of degree prime to five (rather than of odd degree) over $F$ and each $F_P$.
\item \label{valuations E8}
Assume in addition that $K$ is equicharacteristic and $\mX$ is regular.  Then the assertions in parts~(\ref{odd degree E8}) and~(\ref{prime to 5 E8}) hold with the fields $F_P$ replaced by the fields $F_v$ for all $v \in \Omega_{\wh X}$.
\end{enumerate}
\end{prop}

\begin{proof}
For the forward implications, observe that if $G$ is split over a finite extension $E/F$ of degree $d=[E:F]$, and if $F'/F$ is any field extension,
then $G$ also splits over the compositum $E':=EF'$ in an algebraic closure of $F'$, and $[E':F']$ divides $d$.  Taking $F'$ equal to $F_P$ or $F_v$ yields the forward implications.  We now show the reverse implications.

\smallskip

\textit{Proof of~(\ref{odd degree E8}) and the corresponding part of~(\ref{valuations E8})}:

Let $G_0$ be a split simple algebraic group over $F$ of type $E_8$.  Then $H^1(F,G_0)$ classifies simple algebraic
groups of type $E_8$ over $F$, since $G_0 = \Aut(G_0)$.  Given a group $G$ as in the proposition, 
let $[G]$ be the class of $G$ in $H^1(F,G_0)$, and let 
$r_G := R_{G_0}([G])$ be the associated
Rost invariant, say with order $m$.  
Thus $r_G \in H^3(F,\zmod m 2)$.

For each $P \in X$, the group $G$ becomes split over some extension
$E_P/F_P$ of odd degree $d_P$.  Thus the Rost invariant of $G$ over
$F_P$ maps to zero in $H^3(E_P,\Z/m\Z(2))$, and hence it is 
$d_P$-torsion in $H^3(F_P,\Z/m\Z(2))$ by a standard
restriction-corestriction argument.  Thus it is also $d_P'$-torsion, where $d_P'$ is the greatest common divisor of $d_P$ and $m$.  Let $d$ be the least common 
multiple of the odd integers $d_P'$, each of which divides $m$.  
Thus $dr_G  \in H^3(F,\Z/m\Z(2))$
has trivial image in $H^3(F_P,\Z/m\Z(2))$ for all $P$.  It follows
from Theorem~\ref{shapoints}(\ref{sha pts cyclic}) that $dr_G$ is
trivial.  Hence the order of the Rost invariant $r_G$ over $F$ is odd.  

Let $H^1(*, G_0)_0 \subseteq H^1(*, G_0)$ be the subset consisting of classes $\alpha$ such that $R_{G_0}(\alpha)$ has odd order.  By the above, this contains $[G]$.
Now by \cite[Corollary~8.7]{Sem}, 
since $\cha(F)=\cha(K)=0$, 
there is a cohomological invariant $u:H^1(*, G_0)_0 \to H^5(*,\Z/2\Z)$ such that for any field extension $E/F$, the invariant $u([G_E])$ vanishes if and only if $G$ splits over a field extension of $E$ of odd degree. 

By functoriality of $u$, the class $u([G])$ maps to $u([G_{F_P}])$ for every $P \in X$. 
But for every $P \in X$, $G_{F_P}$ is split over an extension of odd degree and hence $u([G_{F_P}])$ is trivial in $H^5(F_P,\Z/2\Z)$. 
By Theorem~\ref{shapoints}(\ref{sha pts cyclic}), it follows that $u([G])$ is trivial in $H^5(F,\Z/2\Z)$.  
The conclusion of~(\ref{odd degree E8})  
now follows from the defining property of $u$.

The corresponding part of~(\ref{valuations E8}) is proved in exactly the same way, but with $F_v$ replacing $F_P$ and with Theorem~\ref{sha_vanish} replacing Theorem~\ref{shapoints}(\ref{sha pts cyclic}).

\smallskip

\textit{Proof of~(\ref{prime to 5 E8}) and the corresponding part of~(\ref{valuations E8})}:

By the main theorem in~\cite{Cher}, since $\cha(F) \ne 2,3,5$, the Rost invariant of $G$ over a field extension $E/F$ has trivial image in $H^3(E,\Z/5\Z(2))$ if and only if $G$ splits over some finite extension of $E$ having degree prime to five.  The desired assertion now follows from Proposition~\ref{invariant local-global}, taking $S(E)$ to be the subset of $H^1(E,G_0)$ that consists of elements that split over some field extension of $E$ having degree prime to five, and with $a$ being the restriction to this subset of the Rost invariant modulo $5$.
\end{proof}

\begin{prop}  \label{Albert}
Assume that $\cha(K) \ne 2,3$.  Then Albert algebras over $F$ have each of the following properties if and only if the respective properties hold after base change to $F_P$ for each $P \in X$.  \begin{itemize}
\item The algebra is reduced.
\item The algebra is split.
\item The automorphism group of the algebra is anisotropic.
\item Two reduced algebras are isomorphic.
\end{itemize}
The same holds for base change to $F_v$ for each $v \in \Omega_{\wh X}$, in the case that $K$ is equicharacteristic and $\mX$ is regular.
\end{prop}

\begin{proof}
Albert algebras are classified by $H^1(F,G)$, where $G$ is a split simple linear algebraic group over $F$ of type $F_4$.  Moreover (see~\cite[9.2,9.3]{Se:Bo}) there are cohomological invariants 
\[f_3:H^1(F,G) \to H^3(F,\Z/2\Z), \ f_5:H^1(F,G) \to H^5(F,\Z/2\Z), \ 
g_3:H^1(F,G) \to H^3(F,\Z/3\Z),\]
where $H^3(F,\Z/3\Z) = H^3(F,\Z/3\Z(2))$.
The properties of Albert algebras listed in the proposition are respectively equivalent to the following conditions involving these invariants (see~\cite[9.4]{Se:Bo}):
\begin{itemize}
\item The invariant $g_3$ vanishes on the algebra.
\item The invariants $f_3$ and $g_3$ each vanish on the algebra.
\item The invariants $f_5$ and $g_3$ are each non-vanishing on the algebra.
\item The two reduced algebras have the same pair of invariants $f_3, f_5$.
\end{itemize}
By the injectivity of the local-global maps on $H^3(F,\Z/2\Z)$, $H^5(F,\Z/2\Z)$, and $H^3(F,\Z/3\Z(2))$ (viz.\ by Theorems~\ref{shapoints}(\ref{sha pts cyclic}) and~\ref{sha_vanish} respectively), and by the functoriality of the invariants $f_3, f_5, g_3$, the assertion then follows.
\end{proof}

%%%%%%%%%%%%%%%%%%%

\medskip

\noindent Author information:

\medskip

\noindent David Harbater: Department of Mathematics, University of Pennsylvania, Philadelphia, PA 19104-6395, USA\\ email: {\tt harbater@math.upenn.edu}

\medskip

\noindent Julia Hartmann: Lehrstuhl f\"ur Mathematik (Algebra), RWTH Aachen University, 52056 Aachen, Germany\\ email:  {\tt 
Hartmann@mathA.rwth-aachen.de}

\medskip

\noindent Daniel Krashen: Department of Mathematics, University of Georgia, Athens, GA 30602, USA\\ email: {\tt dkrashen@math.uga.edu}

\end{document}